\documentclass{article}
\usepackage{amsthm}
\usepackage{bbm,verbatim}

\usepackage[a4paper,margin=36mm]{geometry}

\usepackage{lipsum}
\usepackage{graphicx}
\usepackage{epstopdf}

\usepackage{color}
\usepackage{newfloat}
\usepackage[section]{algorithm}
\usepackage{algorithmic}
\usepackage{booktabs}
\usepackage{caption}
\usepackage{amssymb,amsfonts,amsmath}
\usepackage[shortlabels]{enumitem}
\usepackage{mathtools}
\usepackage{xifthen}
\usepackage{dcolumn}
\usepackage[group-separator={,},detect-all]{siunitx}
\usepackage{subcaption}
\usepackage{hyperref}
\usepackage[capitalize]{cleveref} 
\AtBeginDocument{
	\crefformat{section}{#2\S#1#3} 
	\Crefformat{section}{#2\S#1#3} 
	\crefformat{subsection}{#2\S#1#3} 
	\Crefformat{subsection}{#2\S#1#3} 
	\crefformat{subsubsection}{#2\S#1#3} 
	\Crefformat{subsubsection}{#2\S#1#3} 
}

\Crefname{ALC@unique}{Line}{Lines}

\newcommand{\M}{{\mathcal M}}

\newcommand{\T}{{\mathcal T}}
\newcommand{\R}{\mathbb{R}}
\newcommand{\C}{\mathbb{C}}

\DeclareMathOperator{\diag}{diag}

\DeclareMathOperator{\rank}{rank}

\renewcommand{\imath}{\mathbf{i}}
\renewcommand{\Re}{\mathrm{Re}\,}
\renewcommand{\Im}{\mathrm{Im}\,}

\def\contranssym{\mathsf{H}}

\newcommand{\ct}[1][]{^{{#1}\contranssym}}

\def\matlab{MATLAB}

\numberwithin{equation}{section}
\def\beq{\begin{equation}}
\def\eeq{\end{equation}}
\def\beqs{\begin{equation*}}
\def\eeqs{\end{equation*}}
\def\bseq{\begin{subequations}}
\def\eseq{\end{subequations}}

\def\bigO{\mathcal{O}}
\def\hinf{\mathcal{H}_\infty}
\def\eps{\varepsilon}
\def\Done{\mathcal{D}_1}
\def\Dtwo{\mathcal{D}_2}
\def\lb{_\mathrm{lb}}
\def\ub{_\mathrm{ub}}
\def\opt{_\star}

\def\Xilb{\Xi_\mathrm{lb}}
\def\Xiub{\Xi_\mathrm{ub}}

\def\eiw{\eul^{\imath \omega}}
\def\eiww{\eul^{2 \imath \omega}}

\newcommand{\pderiv}[3][]{
	\ifthenelse{\isempty{#1}}{
		\tfrac{\partial #2}{\partial #3}
		}{
		\tfrac{\partial^{#1} #2}{\partial #3^{#1}}
		}
}

\def\fro{_\mathsf{F}}

\newcommand{\dw}[1][]{\pderiv[#1]{}{\omega}}
\newcommand{\dxi}[1][]{\pderiv[#1]{}{\xi}}

\def\Tc{\T_\xi(\imath \omega)}
\def\Td{\T_\xi(\eiw)}
\def\Pc{\Phi_\xi(\imath \omega)}

\def\dxiTd{\pderiv{\Td}{\xi}}

\DeclareFloatingEnvironment[placement=htb]{algfloat}
\newcommand{\algnote}[1]{\footnotesize \sc{Note: \it#1 } }

\newcounter{subroutine}


\newcommand{\eg}{e.g.,\ }
\newcommand{\ie}{i.e.,\ }
\newcommand{\eul}{\mathrm{e}}

\newcommand{\Hn}{ {\mathbb{H}_n} }   
  
\numberwithin{equation}{section}

\usepackage{sectsty}
\sectionfont{\large}
\subsectionfont{\large}

\ifpdf
\hypersetup{
  pdftitle={Root-max Problems, Hybrid Expansion-Contraction, and Quadratically Convergent Optimization of Passive Systems},
  pdfauthor={T. Mitchell and P. Van Dooren}
}
\fi

\title{Root-max Problems, Hybrid Expansion-Contraction, and Quadratically Convergent Optimization \\of Passive Systems}

\author{Tim Mitchell\thanks{
Max Planck Institute for Dynamics of Complex Technical Systems, Magdeburg, 39106 Germany, 
\href{mailto:mitchell@mpi-magdeburg.mpg.de}{\texttt{mitchell@mpi-magdeburg.mpg.de}}, ORCID: 0000-0002-8426-0242.}
\and
Paul Van Dooren\thanks{
Department of Mathematical Engineering, Universit\'e catholique de Louvain, Louvain-La-Neuve, Belgium.
\href{mailto:paul.vandooren@uclouvain.be}{\texttt{paul.vandooren@uclouvain.be}}, ORCID: 0000-0002-0115-9932.
Visiting the Max Planck Institute of Magdeburg.
} 
}
\date{{\small September 2, 2021\\Revised: May 24, 2022}}

\theoremstyle{plain}
\newtheorem{theorem}{Theorem}[section]
\newtheorem{lemma}[theorem]{Lemma}
\newtheorem{corollary}[theorem]{Corollary}

\newtheorem{assumption}[theorem]{Assumption}
\newtheorem{remark}[theorem]{Remark}

\newtheorem{definition}[theorem]{Definition}

\def\hecscale{0.35}
\def\threescale{0.31}
\def\twoscale{0.35}

\begin{document}
\maketitle

\begin{abstract}
We present quadratically convergent algorithms to compute the extremal value of a 
real parameter for which a given rational transfer function of a linear time-invariant 
system is passive. This problem is formulated for both continuous-time and discrete-time
systems and is linked to the problem of finding a realization of a rational transfer function 
such that its passivity radius is maximized. Our new methods make use of the 
Hybrid Expansion-Contraction algorithm, which we extend and generalize
to the setting of what we call root-max problems.
\end{abstract}



\section{Introduction}
 \label{sec:intro} 

Robustness measures play an important role in systems and control. They provide margins for the perturbations that one can allow on a given 
nominal dynamical system such that the perturbed system still performs as desired. A classical example of such a measure is the so-called 
\emph{distance to instability}~\cite{Van85a}, 
which measures how much one can perturb a stable matrix before destabilization is a possibility. 
A generalization of this is the \emph{complex stability radius} (better known by its reciprocal, the $\hinf$ norm),
which measures how much (complex-valued) uncertainty 
in a dynamical system with input and output can be tolerated before stability is no longer guaranteed \cite{ZhoDG96,HinP05}.
Meanwhile, the \emph{real structured stability radius} and \mbox{$\mu$-value} further restrict the uncertainty 
to be real-valued or structured in a particular sense~\cite{HinP90,HinP90a}.
Such measures are often the subject of optimization in robust control, since it is natural to desire
that  the robustness of models to uncertainty/perturbation be maximized.
Furthermore, in the area of model order reduction, the $\hinf$ norm is one of the main indicators of 
how well a reduced-order surrogate mimics the behavior of a larger 
(and often computationally unwieldy) system~\cite{morGlo84}.
Numerical procedures for computing these robustness measures have been developed in the last few decades and have historically been focused 
on linear time-invariant systems  
described by their generalized state-space model. 

In this paper, we consider a problem that is linked to maximizing the \emph{passivity radius}~\cite{OveV05}, 
which measures how much one can can perturb a passive system before it may lose passivity.
A continuous-time linear time-invariant system is said to be \emph{passive} if the Hermitian part of its transfer function is nonnegative definite in the closed right 
half-plane; see~\cite{Wil71}. When this transfer function is finite dimensional and is described by a state-space model, those conditions can be rephrased in terms of the state-space model parameters. In this paper, we consider 
a transfer function that is dependent on a real parameter and look for an extremal value of the parameter for which this transfer function is still passive. At this extremal value, the parametric transfer function switches from passive to non-passive.
Computation of this extremal parameter value is important, as it allows one to construct certificates
for the passivity of the parametric passive systems.  
As discussed in \cite{MehV20,MehV20a}, 
these certificates play a crucial role in the solution of two important problems:
(i)~finding a realization of a given passive system with optimal passivity radius and  (ii)~finding 
the closest passive system to a given non-passive system.
The first algorithms to compute this extremal value were recently proposed in~\cite{MehV20} and~\cite{MehV20a}, respectively,
for the continuous- and discrete-time cases,
but no convergence analysis was done nor were the methods tested experimentally.
We address these issues, establishing that these methods have at least a superlinear rate of convergence,
while also demonstrating some numerical issues with them.
Most importantly, we present significantly faster new algorithms with
local quadratic convergence and much smaller constant factors in their work complexities.
Our new methods are also more numerically robust than the earlier techniques
and have variants that are both the first algorithms for large-scale use
and guarantee approximations that are locally optimal in a certain sense.  
Finding a nearby passive system to a non-passive one has also been considered in~\cite{GilS18} and \cite{FazGL20}, 
but it was suggested in \cite{MehV20,MehV20a} 
that techniques like the ones we develop here  could be applied to address that problem as well.

A core part of our new methods (from which they derive their quadratic convergence)
is our generalization of the Hybrid Expansion-Contraction (HEC) algorithm.
HEC was first conceived as a way to approximate the $\hinf$ norm of large-scale systems~\cite{Mit14,MitO16}
and was subsequently extended to approximating the real structured stability radius~\cite{GugGMetal17}.
However, HEC and its convergence properties have only been described for these two specific settings,
while the structure of our problem of interest here is quite different.
Unlike the $\hinf$ norm, which is computed by obtaining a global maximizer of a function in one real variable,
the extremal value we consider here for the optimization of passive systems 
is computed by iterating over two real variables.
Consequently, another contribution of this paper is to connect these seemingly disparate things, namely, 
by (i) identifying that all of these problems are actually specific instances 
of what we call \emph{root-max problems} (or equivalently \emph{root-min problems}) 
and (ii) generalizing HEC and its convergence results
to this broad new class.
Besides enabling our new methods here, we hope that our generalization of HEC 
will both increase awareness for identifying root-max and root-min problems and ease facilitation of new HEC-based methods.

The paper is organized as follows.  We first establish notation and preliminary material in~\cref{sec:background}.
Then, in~\cref{sec:hec}, we introduce root-max problems and generalize HEC and its convergence results 
to this problem class.
In~\cref{sec:cont}, we describe the continuous-time version of our passivity radius problem 
and our corresponding new algorithm to solve it, while the discrete-time case is handled in~\cref{sec:disc}.
Numerical experiments and concluding remarks are, respectively, 
given in~\cref{sec:experiments} and~\cref{sec:conclusion}.

\section{Preliminaries}
\label{sec:background}
We begin with notation.  The set of Hermitian matrices in~$\C^{n \times n}$ is denoted by~$\Hn$, 
with $A\succ 0$ ($A\succeq 0$) additionally signifying that $A \in \Hn$ is positive (semi-)definite.
$\Lambda(A)$ denotes the spectrum of a matrix $A$ and, when $A$ is Hermitian, we additionally use 
the shorthand~$\lambda_{\min}(A)$ to denote its smallest eigenvalue.
For any square matrix $A$, $\alpha(A) \coloneqq \max \{ \Re \lambda : \lambda \in \Lambda(A) \}$ 
and $\rho(A) \coloneqq \max \{ | \lambda | : \lambda \in \Lambda(A) \}$ 
respectively denote the \emph{spectral abscissa} and \emph{spectral radius} of $A$;
note that $A$ is continuous-time (discrete-time) \emph{asymptotically stable} if and only if $\alpha(A) < 0$ ($\rho(A) < 1$).
$\Re (Z)$ and $\Im (Z)$, respectively, denote the real and imaginary parts of a complex matrix $Z$, 
while the (conjugate) transpose of a vector or matrix $V$ is denoted by $V^{\mathsf{T}}$ ($V^{\mathsf{H}}$).
We use $I_n$ for the $n \times n$ identity matrix.

The models that we consider here are given by their standard state-space form, which means 
that their associated transfer functions are proper (\ie bounded at infinity).
In the continuous-time setting, the transfer function arises from the \emph{Laplace transform} of the  system
\begin{equation} \label{statespace_c}
\begin{array}{rcl} \dot x(t) & = & Ax(t) + B u(t),\ x(0)=0,\\
y(t)&=& Cx(t)+Du(t),
\end{array}
\end{equation}
where $A\in \C^{n \times n}$,   $B\in \C^{n \times m}$, $C\in \C^{p \times n}$, $D\in \C^{p \times m}$,
 and  $x(\cdot)$, $u(\cdot)$,  $y(\cdot)$  are time-dependent vector-valued functions denoting, respectively, the \emph{state},  \emph{input},
and \emph{output} of the system.  In the discrete-time setting, the transfer function arises from the \emph{z-transform} applied to 
\begin{equation} \label{statespace_d}
\begin{array}{rcl} x_{k+1} & = & Ax_k + B u_k,\ x_0=0,\\
y_k&=& Cx_k+Du_k,
\end{array}
\end{equation}
where now $x_k$,  $u_k$,  and  $y_k$  are vector-valued sequences denoting, 
respectively, the \emph{state}, \emph{input}, and \emph{output} of the system.  
In both cases, we denote these systems by four-tuples of matrices $\M \coloneqq \left\{A,B,C,D\right\}$ and
their associated rational matrices 
\beq
	\label{transfer}
	\T(\lambda) \coloneqq C(\lambda I_n-A)^{-1}B + D
	\quad \text{and} \quad
	\T^\mathsf{H}(\lambda) \coloneqq B^\mathsf{H}(\lambda I_n-A^\mathsf{H})^{-1}C^\mathsf{H}+D^\mathsf{H}
\eeq
are, respectively, the associated \emph{transfer function} and \emph{para-conjugate transfer function},
where the variable $\lambda$ stands for the Laplace variable $s$ in the continuous-time setting and the delay operator $z$ 
in the discrete-time case.

We restrict ourselves to system models which are \emph{minimal}, \ie the pair $(A,B)$ is \emph{controllable} 
(for all~$\lambda\in \mathbb C$, 
$\rank \begin{bmatrix} \lambda I_n-A & B \end{bmatrix} =n$), and the pair $(A,C)$ is \emph{observable} (\ie $(A^\mathsf{H},C^\mathsf{H})$ is controllable). 
If the model is not minimal, one can always construct a minimal realization by 
removing the uncontrollable and unobservable parts, which can be done in a backward stable manner~\cite{Van81b}.
The conditions for passivity can then be expressed in terms of a linear matrix inequality involving the matrices of the system model~$\M$. 
The \emph{passivity margin} of the system
model~$\M$ will then be shown to depend on the extremal value 
of a real parameter~$\xi$ for which a particular parametric system model~$\M_\xi$ loses its passivity property.

For rates of convergence, we use the notion of Q-quadratic and Q-superlinear convergence, 
where “Q” stands for “quotient”; see \cite[p.~619]{NocW99} for more details.

\section{Root-max problems and Hybrid Expansion-Contraction}
\label{sec:hec}
Let $\Done \subseteq \R$ be connected, $\Dtwo \subset \R^{N}$ be compact, 
$g: \Done \times \Dtwo \to \R$ be a continuous function,
and~$g_x : \Done \to \R$ and~$g_\eps : \Dtwo \to \R$ be the following 
restrictions of $g$:
\bseq
	\label{eq:g_funs}
	\begin{alignat}{2}
		g_x(\eps) &\coloneqq g(\eps, x), \qquad && \text{where $x \in \Dtwo$ is fixed}, \\
		g_\eps(x) &\coloneqq g(\eps, x), \qquad && \text{where $\eps \in \Done$ is fixed}.
	\end{alignat}
\eseq
Consider the following root-finding problem, which we call a \emph{root-max problem}:
\beq
	\label{eq:root_max}
	\text{Determine an}~\eps \in \Done : \qquad
	f(\eps) \coloneqq \max_{x \in \mathcal{D}_2} g(\eps, x) 
	= \max_{x \in \mathcal{D}_2} g_\eps(x)
	= 0,
\eeq
where we assume that the functions $g_\eps$ are bounded above for all $\eps \in \Done$.
Suppose that there exists \mbox{$\eps\lb, \eps_0 \in \Done$} with $\eps\lb < \eps_0$ and
either \mbox{$f(\eps\lb) < 0 \leq f(\eps_0)$} or \mbox{$f(\eps\lb) > 0 \geq f(\eps_0)$} holding.
Then, by continuity of~$g$, it is clear that \eqref{eq:root_max} has at least 
one root~\mbox{$\eps\opt \in (\eps\lb,\eps_0] \subseteq \Done$}
such that $f(\eps\opt) = 0$.
Of course, if $f(\eps_0) = 0$ holds, then we can take $\eps\opt = \eps_0$.
For our purposes in this section, it is convenient to assume 
the convention that $f(\eps\lb) < 0$ and~$0 \leq f(\eps_0)$ hold, but 
note that each of these inequalities can be modified to be (non-)strict
or reversed, \eg $f(\eps\lb) \geq 0 > f(\eps_0)$, as desired for a specific setting.

Many well-known distance measures can be written in the form of the root-max problem given by \eqref{eq:root_max},
or equivalently, as a \emph{root-min problem}, where the $\max$ functions in \eqref{eq:root_max} are switched
to $\min$ functions and $g_\eps$ must then be bounded below for any fixed $\eps \in \Done$.
For example, the distance to instability and the real stability radius can be naturally 
expressed as root-max problems.  
As we explain later, our particular problem of interest, the optimization of passive systems, 
also falls in this problem class, although we find it more natural to use the root-min form for that context.

In this section, we show that the Hybrid Expansion-Contraction (HEC) algorithm of~\cite{Mit14,MitO16},
which was originally conceived as a method for approximating the $\hinf$ norm, 
actually generalizes to address the class of root-max and root-min problems 
that we have just defined here.  
We use the root-max form to generalize HEC and its associated convergence properties,
since this maintains consistency with its initial usage as well as its name itself: Hybrid Expansion-Contraction.

\begin{assumption}
We assume that the function $g$ is continuously differentiable.
\label{asm:g_c1}
\end{assumption}

We use \cref{asm:g_c1} to keep this section from becoming significantly more technical.
In~\cref{rem:nonsmooth}, we discuss how the convergence of HEC is not critically reliant upon this smoothness condition
and how it can be weakened.  

\begin{assumption}
For any $\eps \in \Done$, we assume that we can obtain local maximizers or stationary points of $g_\eps$ and do so exactly,
\ie the norm of the gradient is zero,
but finding a global maximizer of $g_\eps$ cannot be guaranteed.  
In other words, we cannot be guaranteed to obtain the value of $f(\eps)$ in practice, but we are 
guaranteed to obtain (generally locally optimal) lower bounds to it, which may or may not agree with the value of $f(\eps)$. 
\label{asm:g_stat}
\end{assumption}

\begin{remark}
Some comments on~\cref{asm:g_stat} are in order.
Guaranteeing convergence to global maximizers of general functions, \eg nonconcave ones,
is typically not possible, and even in special cases where it is, techniques to do so are often prohibitively expensive.
If one could reliably compute the value of~$f(\eps)$, 
then~\eqref{eq:root_max} could simply be solved using standard root-finding techniques with bracketing, \eg regula falsi.
In contrast, as we elucidate below, 
under the much milder and more realistic assumptions given in~\cref{asm:g_stat}, 
standard root-finding techniques can break down
when trying to solve instances of~\eqref{eq:root_max}, precisely because obtaining the value of~$f(\eps)$ is not guaranteed.
Having an algorithm for~\eqref{eq:root_max} that performs robustly and predictably under~\cref{asm:g_stat}
motivated the development of HEC, although it was not until this paper
that HEC was actually considered from this general perspective.
\label{rem:mismatch}
\end{remark}

HEC was borne out of the specific desire for a faster and more reliable alternative to earlier state-of-the-art scalable methods
for approximating the $\hinf$ norm.
Prior to its introduction, Guglielmi, G\"urb\"uzbalaban, and Overton~\cite{GugGO13} had proposed an
$\hinf$-norm approximation method that attempts to compute the unique root of a particular 
monotonically increasing function in one real variable; the reciprocal of this root is the $\hinf$ norm. 
The main wrinkle here is that with existing techniques, 
evaluating the function to guaranteed accuracy would actually be more expensive than computing the $\hinf$ norm directly, 
but crucially, Guglielmi et al.\@ devised a powerful, scalable subroutine that efficiently computes a lower bound to the function value,
which in practice, also often coincides with the true function value.  Hence, they proposed using their fast subroutine inside 
a Newton-bisection-based outer iteration in order to compute the root of this function, and thus in turn, the $\hinf$ norm.
However, per~\cite[Acknowledgements]{GugGO13}, as first observed by Mitchell, the first author here,
this root-finding approach of~Guglielmi et al.\@ actually can sometimes break down, converging to arbitrarily bad approximations
to the $\hinf$ norm that are not even locally optimal.
 Moreover, when this breakdown happens,
their algorithm's typically fast local rate of convergence also degrades to linear.

In line with~\cref{rem:mismatch}, these breakdowns arise 
precisely because the function whose root is sought is not guaranteed to be computed accurately,
and so using a standard root-finding method as the outer iteration is fraught with danger.
The crux of the matter in the method of Guglielmi et al.\@ is that the sometimes inaccurate estimates 
for the function values can cause the bracket containing the root to
be incorrectly and irrevocably updated, which in~\cite{Mit14,MitO16} was coined a \emph{bound mismatch error};
for a full description of this how comes about, see~\cite[Section~3.2]{MitO16}.
HEC overcomes this critical problem by instead employing one-sided convergence,
which was motivated by a key observation~\cite[p.~994]{MitO16}: 
when only lower bounds to the (recall monotonically increasing) function are guaranteed, if the computed estimate is negative, 
the direction of the unique root cannot be determined, but if the computed estimate is positive, one
does know that the root lies to the left.  
Thus, the HEC algorithm was designed to compute a decreasing sequence of upper bounds in order
to converge to a root.

In applications of HEC explored so far~\cite{MitO16,GugGMetal17} where~\cref{asm:g_stat} holds,
HEC often converges to roots of specific instances of~\eqref{eq:root_max},
which is guaranteed if~$f(\eps)$ can always be computed accurately.
However, under~\cref{asm:g_stat} as stated,
HEC instead guarantees convergence to what we call a \emph{pseudoroot} of~\eqref{eq:root_max}, 
which is either an actual root of~\eqref{eq:root_max}, or, roughly speaking, a locally optimal approximation to one;
we will define this notion exactly momentarily.
In~\cite{Mit14,MitO16}, no name was given for this concept as it was not considered
in that context.  

Although HEC uses bracketing,
the facts that it (i) only ever updates its upper bound 
and (ii) deliberately uses one-sided convergence to roots
make HEC strikingly different to other root-finding methods.
But, as the upcoming theoretical results will clarify,
HEC's one-sided convergence does not come at the cost of sacrificing fast local convergence;
under mild smoothness assumptions, the local rate of convergence of HEC is at least quadratic.
In~\cite[p.~997]{MitO16}, HEC is described as an ``adaptively positively or negatively damped Newton method",
which means that HEC sometimes takes steps smaller or larger, 
respectively, than the regular Newton step.
It may seem unintuitive, but the local rate of convergence of HEC does remain at least quadratic 
even when HEC takes smaller steps (positive damping).
Meanwhile, the ability of HEC to take larger steps (negative damping) is a good thing.
Suppose that HEC converges to a root $\tilde \eps$.\footnote{When $f$ has
multiple roots, negative damping may bias HEC towards finding roots closer to $\eps\lb$ before its local convergence behavior sets in.}  Due to the one-sided convergence,
a step larger than the Newton one can never overshoot $\tilde \eps$, and so larger steps will always 
make more progress towards $\tilde \eps$ than the corresponding Newton steps would.
Thus, when negative damping is frequent, HEC can be faster than Newton's method.
For illustrations of positive and negative damping, please see~\cite[Fig.~4]{MitO16}.

\subsection{The generalized HEC algorithm and its convergence properties}
\label{sec:hec_generic}
Having put HEC and its properties in the context of its own history and root finding,  
we now set to the task of precisely describing how HEC actually works and
generalizing it to root-max problems~\eqref{eq:root_max}.
The convergence properties that we establish for our generalized version of HEC 
are, at a very high level, proved using similar arguments to those given by the first author here and Overton in~\cite[Section~4]{MitO16}
for the specific case of approximating the $\hinf$ norm.
However, our generalization here makes these convergence 
results far more accessible in terms of being much easier to both understand and apply far more broadly.

\begin{definition}
Given $\tilde \eps \in \Done$ and $\tilde x \in \Dtwo$, 
$(\tilde \eps, \tilde x)$ is a \emph{pseudoroot of \eqref{eq:root_max}}
if~$g(\tilde \eps, \tilde x) = 0$ and~$\tilde x$ is a stationary point of $g_{\tilde\eps}$.
\end{definition}

Defining pseudoroot in terms of a stationary point of $g_{\tilde\eps}$,
as opposed to a local maximizer, which might seem
more intuitive, is intentional.  The reason for this is subtle and requires more context to explain,
so we defer this discussion to~\cref{rem:stat_pts}.
As we see in the following simple result (whose proof we omit as it is elementary), 
pseudoroots are intimately related with roots of~\eqref{eq:root_max}.

\begin{lemma}
Let $\tilde \eps \in \Done$, $\tilde x \in \Dtwo$, 
and $(\tilde \eps, \tilde x)$ be a pseudoroot of~\eqref{eq:root_max}.
Then  
$\tilde \eps$ is a root of~\eqref{eq:root_max} if and only if $\tilde x$
is a global maximizer of $g_{\tilde\eps}$.
Otherwise, $0 < f(\tilde \eps)$.
\end{lemma}

As subroutines, HEC requires both a root-finding method with bracketing and optimization solver,
and we assume these subroutines have the following properties.  

\begin{assumption}
We assume that the root-finding and optimization subroutines used by HEC
are deterministic, \ie they return the same answer for the same initial data,
converge exactly (see also \cref{asm:g_stat}), and the root-finding method uses bracketing to ensure convergence to a root,
while the optimization solver is monotonic, \ie 
it always increases the value of the objective function being maximized at successive iterates
until it reaches a stationary point.
\label{asm:subs}
\end{assumption}

Many root-finding methods use bracketing, while unconstrained optimization solvers are typically monotonic by design.
Most solvers for these problems are also deterministic, and so this set of assumptions is mild.
The remaining assumption that the subroutines converge exactly does not hold in inexact arithmetic,
but this assumption is only used to establish our theoretical results.  
In practice, good implementations of HEC behave as the theory predicts as long as the subroutines 
are reasonably accurate.

\begin{figure}[t]
\centering
\includegraphics[scale=\hecscale,trim=1.0cm 1.75cm 0cm 0cm,clip]{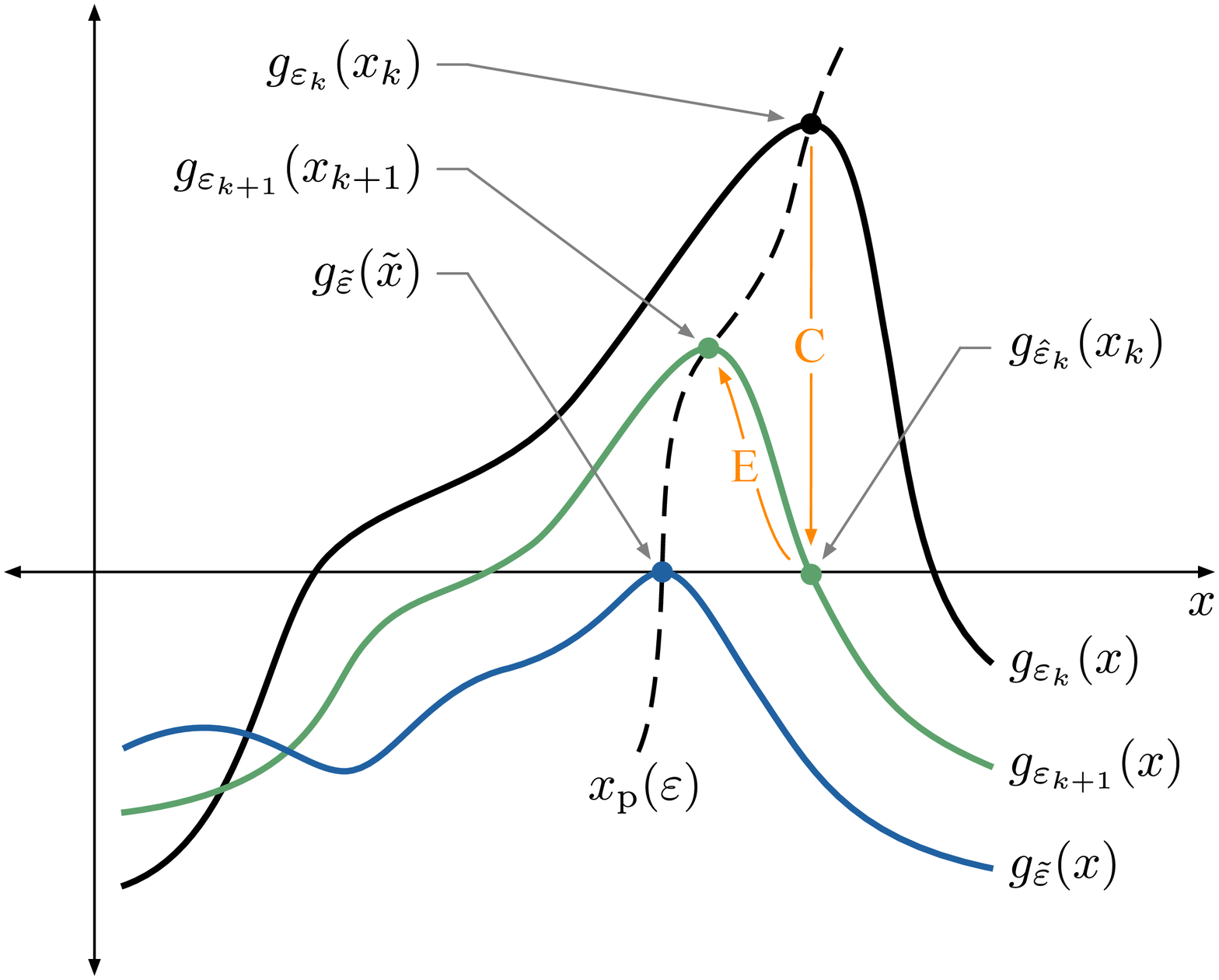}
\caption{Illustration of HEC of converging to a pseudoroot $(\tilde \eps,\tilde x)$ of~\eqref{eq:root_max},
where $\tilde \eps$ is also a root of~\eqref{eq:root_max}.
At iteration $k$, HEC has found the global maximizer $x_k$ of~$g_{\eps_k}$.
The contraction phase, denoted by `C' above, computes $\eps_{k+1} = \hat \eps_k \leq \eps_k$,
which is a root of~$g_{x_k}$ in the interval $(\eps\lb,\eps_k]$.
The subsequent expansion phase, denoted by `E' above, then computes the global maximizer~$x_{k+1}$
of~$g_{\eps_{k+1}}$.
The dashed curve~$x_\mathrm{p} : \Done \to \Dtwo$ denotes a continuous path of maximizers 
of~$g_\eps$ as~$\eps$ is varied, 
where $x_\mathrm{p}(\tilde \eps) = \tilde x$ corresponds to the pseudoroot~$(\tilde \eps,\tilde x)$.
Sufficiently close to~$(\tilde \eps,\tilde x)$, it is typical that HEC only encounters a single path $x_\mathrm{p}$
as depicted here.
}
 \label{fig:hec}
\end{figure}
We now are ready to describe how the (generalized) HEC algorithm works.
As we define the algorithm here, readers may also wish to refer to~\cref{fig:hec},
which illustrates HEC converging to a pseudoroot $(\tilde \eps, \tilde x)$, 
where~$\tilde \eps$ is indeed a root of~$f$ and~$\tilde x$ is a maximizer of~$g_{\tilde \eps}$.
By construction, HEC generates a monotonically decreasing sequence~$\{\eps_k\} \to \tilde \eps$.
For~\mbox{$x_k \in \Dtwo$} fixed with $g_{x_k}(\eps_k) \geq 0$, first note that we have
\[
	g_{x_k}(\eps\lb) 
	\leq 
	f(\eps\lb) 
	< 0 \leq g_{x_k}(\eps_k) \leq f(\eps_k).
\]
The \emph{one-parameter contraction phase} reduces $\eps_k$ by finding a 
root $\hat \eps_k \in (\eps\lb,\eps_k]$ of~$g_{x_k}$.
By the inequalities above, there must be at least one root in this bracket.
If~\mbox{$g_{x_k}(\eps_k) = 0$}, the contraction phase simply returns $\hat \eps_k = \eps_k$.
Otherwise, bisection can be used to find a root in $(\eps\lb,\eps_k)$.
But  if $g_{x_k}$ is sufficiently smooth at~$\hat \eps_k$, then, 
\eg Newton's or Halley's method could find it with far fewer iterations.  
Of course, these faster root-finding methods are not guaranteed to converge and $g_{x_k}$ may not 
be always be sufficiently smooth, which is why, per~\cref{asm:subs}, it is important to combine both approaches, 
\eg Newton-bisection.
Bracketing and bisection ensure convergence to a root of $g_{x_k}$,
but the ability to also take Newton (or Halley) steps, assuming that they fall inside the current bracket, 
can yield quadratic (or cubic) convergence when sufficient smoothness holds.
Subsequently, for~$\hat \eps_k \in \Done$ now fixed and $g_{\hat \eps_k}(x_k) = 0$,
the \emph{multi-parameter expansion phase} attempts to maximize $g_{\hat \eps_k}$ by initializing 
an optimization solver at~$x_k$.  
If optimization returns~\mbox{$x_{k+1} = x_k$}, there is nothing to do, 
\eg when~$x_k$ is a stationary point of $g_{\hat \eps_k}$.
Otherwise, since the optimization solver is monotonic by~\cref{asm:subs},
the solver must converge to a stationary point~$x_{k+1}$ of~$g_{\hat \eps_k}$ (typically a maximizer) 
such that $g_{\hat \eps_k}(x_{k+1}) > 0$.
Beyond the conditions in~\cref{asm:subs}, HEC does not specify a specific optimization method,
though fast methods should be used when possible.
This process of alternating between root finding (contraction) and 
optimization (expansion) is repeated in a loop and it converges to a pseudoroot of~\eqref{eq:root_max}.
Pseudocode for HEC is given~\cref{alg:hec}.

\begin{algfloat}[t]
\begin{algorithm}[H]
\caption{Hybrid Expansion-Contraction (HEC)}
\label{alg:hec}
\begin{algorithmic}[1]
	\setcounter{ALC@unique}{0}
	\REQUIRE{ 
		$\eps\lb, \eps_0 \in \Done$ and $x_0 \in \Dtwo$ 
		such that $f(\eps\lb) < 0 \leq g(\eps_0, x_0) \leq f(\eps_0)$ \\ 
	}
	\ENSURE{ 
		$(\tilde \eps, \tilde x)$ such that $g(\tilde \eps, \tilde x) = 0$ and $\tilde x$ is a stationary point of $g_{\tilde \eps}$
		\\ \quad
	}
	
	\FOR { $k = 0,1,2,\ldots$ }
		\STATE \COMMENT{Contraction: deterministic root-finding method initialized at $\eps_k$}
		\STATE $\hat \eps_k \gets$ a root of $g_{x_k}$ with $\hat \eps_k \in (\eps\lb,\eps_k]$
		\IF { $x_k$ is a stationary point of $g_{\hat \eps_{k}}$ }
			\label{line:hec_stat_pt}
			\STATE $(\tilde \eps, \tilde x) \gets (\hat \eps_k, x_k)$
			\RETURN
		\ENDIF
		\STATE \COMMENT{Expansion: deterministic optimization method initialized at $x_k$}
		\STATE $x_{k+1} \gets$ a stationary point of $g_{\hat \eps_k}$
			 with $g_{\hat \eps_k}(x_{k+1}) > g_{\hat \eps_k}(x_k)$
 		\STATE $\eps_{k+1} \gets \hat\eps_k$
	\ENDFOR 
\end{algorithmic}
\end{algorithm}
\vspace{-0.4cm}
\algnote{
If the conditional statement in~\cref{line:hec_stat_pt} is never satisfied, 
then by~\cref{thm:hec_converge}, 
HEC produces two infinite sequences $\{\eps_k\}$ and $\{x_k\}$, 
with the former converging to $\tilde \eps$ and the latter having at least one cluster point,
any of which we denote as $\tilde x$.
Contraction must use a root-finding method
with bracketing, e.g., Newton-bisection, 
to ensure a root of~$g_{x_k}$ in the given bracket $(\eps\lb,\eps_k]$ is found.
The inequality in the expansion phase is guaranteed by simply 
initializing optimization at $x_k$ and using a monotonic optimization solver.
Finally, HEC can begin with either an expansion or contraction phase,
and which is more convenient may depend on the particular application.
}
\end{algfloat}

\begin{theorem}[Convergence of HEC]
\label{thm:hec_converge}
Under~\cref{asm:g_c1,asm:subs} and given valid initial data, \cref{alg:hec} generates the sequences $\{\eps_k\}$ 
converging monotonically to a limit $\tilde \eps$ and~$\{x_k\}$ with at least one cluster point,
where $(\tilde \eps, \tilde x)$ is a pseudoroot of \eqref{eq:root_max}.
\end{theorem}
\begin{proof}
We assume that conditional statement in~\cref{line:hec_stat_pt} of~\cref{alg:hec} is never met,
as otherwise the theorem clearly holds.  
Since the algorithm ensures that $\{\eps_k\}$ is a monotonically decreasing sequence that is bounded below by $\eps\lb$,
it must converge to a limit~$\tilde \eps$, 
and so it follows that $\lim_{k\to\infty} \hat\eps_k = \lim_{k\to\infty} \eps_{k+1} = \tilde \eps$ as well.
By construction, for all~$k \geq 1$, 
the algorithm also ensures that $g_{\eps_k}(x_k) > 0$ with $x_k$ being a stationary point of~$g_{\eps_k}$.
Now suppose that $\lim_{k\to\infty}g(\eps_k, x_k) \neq 0$.
Then there is a subsequence~$\{x_{k_i}\}$ for which $\{g(\eps_{k_i}, x_{k_i})\}$ is bounded below by some~$\gamma > 0$.
Thus, by taking a further subsequence if necessary,
we may assume without loss of generality that $\{x_{k_i}\}$ converges to a limit $\tilde x$.
By continuity of $g$, it follows that~$\{g(\eps_{k_i}, x_{k_i})\}$ converges to~$g(\tilde \eps, \tilde x) \geq \gamma$.
However, since~$\{\hat \eps_{k_i}\}$ also converges to $\hat \eps$, then $\{g(\hat \eps_{k_i}, x_{k_i})\}$ must converge to the same limit $g(\tilde \eps, \tilde x)$,
which is a contradiction, since by definition of the contraction step, $g_{x_{k_i}}(\hat\eps_{k_i}) = 0$ must hold for all $i$.
Thus, $\lim_{k\to\infty} g(\eps_k, x_k) = 0$ must hold.
Although~$\{x_k\}$ may not converge, the sequence is bounded since~$\Dtwo$ is a compact subset of $\R^N$,
and so $\{x_k\}$ must have at least one cluster point.
As~$\| \nabla g_{\eps_k} (x_k) \| = 0$ holds for all $k \geq 1$, 
clearly $\| \nabla g_{\tilde \eps} (\tilde x) \| = 0$ also holds,
and so $\tilde x$ is also a stationary point of~$g_{\tilde \eps}$,
hence~$(\tilde \eps, \tilde x)$ is a pseudoroot of~\eqref{eq:root_max}.
\end{proof}

\begin{remark}
\label{rem:stat_pts}
Stationary points of $g_{\hat \eps_k}$ computed in the expansion phases 
will typically be maximizers, and some optimization solvers 
can guarantee convergence to maximizers (under appropriate assumptions).
However, while \cref{thm:hec_converge} guarantees
that HEC converges to a pseudoroot~$(\tilde \eps,\tilde x)$ of \eqref{eq:root_max},
it does not guarantee that $\tilde x$ is a local maximizer of~$g_{\tilde \eps}$,
just that it is a stationary point.
Nevertheless, whenever the expansion phases consistently return local maximizers,
we do observe in practice that~$\tilde x$ is also a local maximizer; see~\cite{MitO16,GugGMetal17}. 
While it seems unlikely that~$\tilde x$ would only be stationary, we do not believe it is impossible;
e.g., it is easy to imagine that the functions~$g_{\eps_k}$ shown 
in~\cref{fig:hec} could instead converge to a function~$g_{\tilde \eps}$
that is constant in an interval about~$\tilde x$.
\end{remark}

\begin{remark}
\label{rem:nonsmooth}
It is only in the last sentence of the proof of~\cref{thm:hec_converge} that~\cref{asm:g_c1}
is used.  However, \cref{thm:hec_converge} can be extended to functions $g_\eps$
that have some nonsmoothness, \eg at maximizers, if one instead 
uses a concept of stationarity that can both handle nonsmooth points and remains
continuous so that the limit argument in the proof still holds.  
\end{remark}

\begin{figure}[t]
\centering
\includegraphics[scale=\hecscale,trim=1.0cm 2.0cm 0cm 0cm,clip]{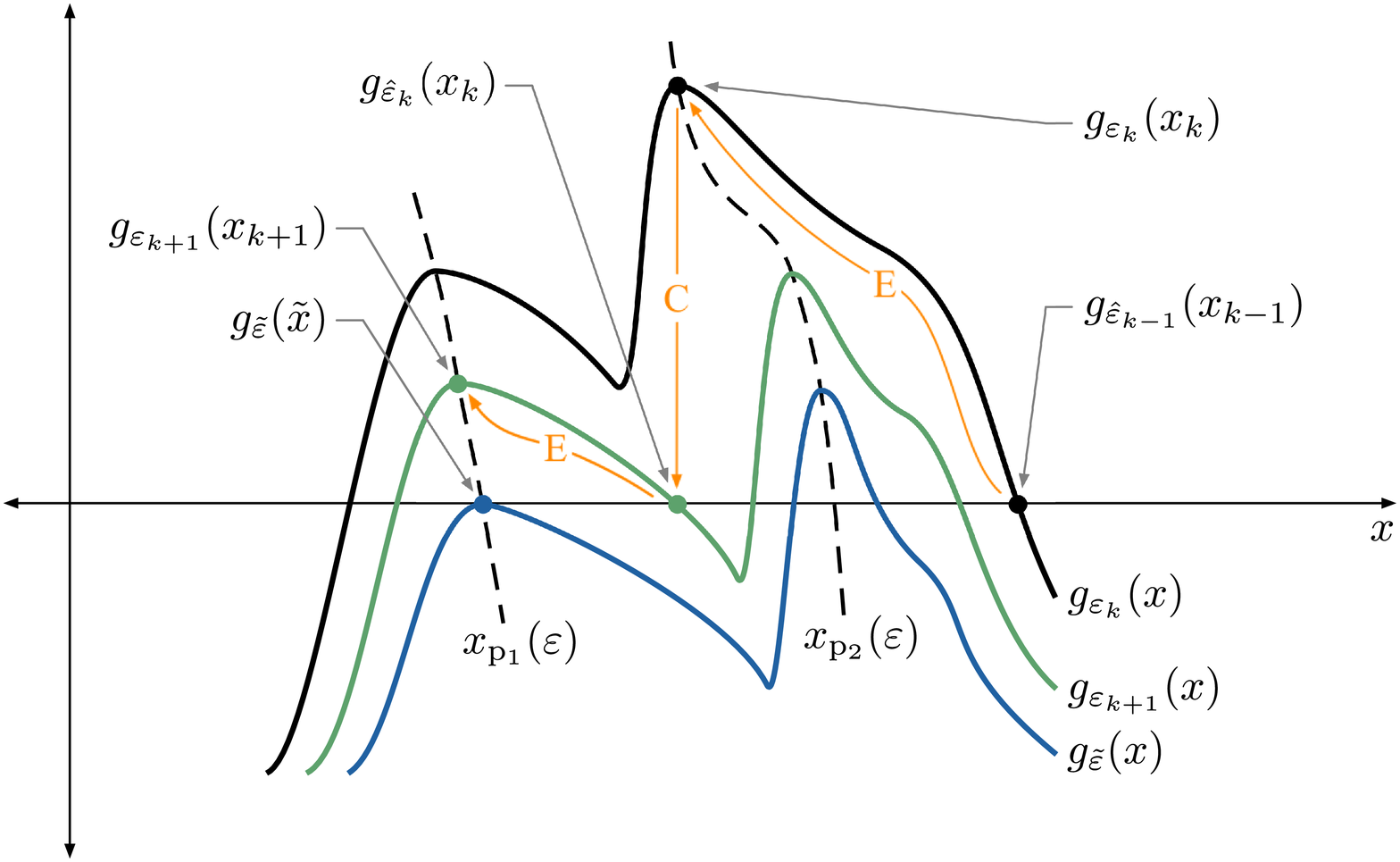}
\caption{Illustration of HEC of encountering two different paths of stationary points,  
$x_{\mathrm{p}_1} : \Done \to \Dtwo$ and  $x_{\mathrm{p}_2} : \Done \to \Dtwo$,
with HEC eventually converging to a pseudoroot~$(\tilde \eps,\tilde x)$ of~\eqref{eq:root_max} 
on path $x_{\mathrm{p}_1}$, but $\tilde \eps$ is not a root of~\eqref{eq:root_max}.
Note that paths of stationary points do not necessarily need to contain a pseudoroot 
(intersect with the $x$-axis), and although HEC may encounter and/or oscillate between multiple such paths
as the algorithm converges,  
this does not affect the convergence result for HEC described by~\cref{thm:hec_converge}.
For more details, see the caption of~\cref{fig:hec}.
}
\label{fig:hec2}
\end{figure}

Although by construction~\cref{alg:hec} produces a monotonically decreasing sequence $\{\eps_k\}$,
note that the sequence $\{g_{\eps_k}(x_k)\}$ produced by the expansion phases
is not necessarily monotonic, even though
it must converge to zero.
For example, if the highest two curves in~\cref{fig:hec} were to instead cross each other to the left and right of the continuous 
path of global maximizers $x_\mathrm{p}$, then $g_{\eps_{k+1}}(x_{k+1}) > g_{\eps_k}(x_k)$ would hold.
Moreover, non-monotonicity of~$\{g_{\eps_k}(x_k)\}$ can also result from HEC encountering multiple such paths of 
stationary points as it progresses.
These paths can consist of global or local maximizers or sometimes even both.
\cref{fig:hec2} shows a depiction where~$x_{\mathrm{p}_1}$ and~$x_{\mathrm{p}_2}$ 
are two separate continuous paths of local maximizers of $g_\eps$  and
HEC encounters both paths, but in this illustration, $\{g_{\eps_k}(x_k)\}$ is monotonically converging to zero.
Again, encountering multiple such paths does not affect the convergence
result of~\cref{thm:hec_converge}.
However, to show that the sequence~$\{\eps_k\}$ generated by HEC converges quadratically to $\tilde \eps$,
it will be simpler to assume that HEC eventually only encounters a single continuous path of local maximizers,
like as is shown in~\cref{fig:hec}.

\begin{theorem}[Quadratic convergence of HEC]
\label{thm:hec_quad}
Suppose that~\cref{asm:g_c1,asm:subs} hold, and so with valid initial data,  
\cref{alg:hec} converges as described in~\cref{thm:hec_converge}.
Additionally suppose that the sequence $\{x_k\}$ only has a single cluster point $\tilde x$,
and $\tilde x$ lies on an open continuous path $x_\mathrm{p} : \Done \to \Dtwo$ of stationary points of~$g_\eps$
as $\eps$ varies with~$\tilde x = x_\mathrm{p}(\tilde \eps)$.
If~$g$ and~$x_\mathrm{p}$ are twice continuously differentiable
at~$(\tilde \eps, \tilde x)$ and $\tilde \eps$, respectively, 
$\tilde x$ is a local maximizer of $g_{\tilde \eps}$, and $g_{\tilde x}^\prime(\tilde \eps) \neq 0$,
then the sequence~$\{\eps_k\}$ converges Q-quadratically to $\tilde \eps$.
\end{theorem}
\begin{proof}
We begin by defining the function
\beq
 	h(\eps) \coloneqq g(\eps, x_\mathrm{p}(\eps)).
\eeq
Note that $h(\tilde \eps) = g(\tilde \eps, \tilde x) =  0$ as $(\tilde \eps, \tilde x)$ is a pseudoroot of~\cref{eq:root_max}.
Since the sequence~$\{x_k\}$ only has one cluster point,
there also exists some~$K$ such that for all~$k \geq K$, all of the following properties hold:
\begin{enumerate}[label=(\roman*),font=\normalfont]
\item point $x_k$ lies on path $x_\mathrm{p}$ with $x_k = x_\mathrm{p}(\eps_k)$, and so $h(\eps_k) = g_{x_k}(\eps_k)$,
\item $h^\prime(\eps_k) = g_{x_k}^\prime(\eps_k) \neq 0$,
\item $h$ and $g_{x_k}$ are twice continuously differentiable at $\eps_k$.
\end{enumerate}
By our assumptions, all of these statements also hold at~$\tilde \eps$.
The agreement of the first derivatives in (ii)
follows from the envelope theorem (or more generally, \cite[Theorem~10.31]{RocW98}),
since maximizers of $g_\eps$ do not occur at $\infty$ as $\Dtwo$ is compact.

Having established the needed properties above, 
we now consider the corresponding Newton steps for $h$ and 
$g_{x_k}$ evaluated at $\eps_k$, which also must coincide, \ie
\beq
	\label{eq:hec_newton}
	\eps^\mathrm{N}_k \coloneqq \eps_k - \frac{h(\eps_k)}{h^\prime(\eps_k)} = \eps_k - \frac{g_{x_k}(\eps_k)}{g_{x_k}^\prime(\eps_k)}.
\eeq
However, \cref{alg:hec} sets $\eps_{k+1} \coloneqq \hat \eps_k$, where $g_{x_k}(\hat \eps_k) = 0$.
Separately applying Taylor's theorem to $h$ and $g_{x_k}$, we have that
\[
	0 = h(\tilde \eps) = h(\eps_k) + h^\prime(\eps_k)(\tilde \eps - \eps_k) + \tfrac{1}{2} h^{\prime\prime} (\xi_k) (\tilde \eps - \eps_k)^2
\]
for some $\xi_k \in  [\tilde \eps, \eps_k]$ and
\[
	0 = g_{x_k}(\hat \eps_k) = g_{x_k}(\eps_k) + g_{x_k}^\prime(\eps_k)(\hat \eps_k - \eps_k) + \tfrac{1}{2} g_{x_k}^{\prime\prime} (\eta_k) (\hat \eps_k - \eps_k)^2
\]
for some $\eta_k \in [\tilde \eps, \eps_k]$.
Respectively dividing the two equations above by $h^\prime(\eps_k)$ and~$g_{x_k}^\prime(\eps_k)$, 
and then subtracting the first from the second and using \eqref{eq:hec_newton} along with~$\eps_{k+1} = \hat \eps_k$, 
we obtain
\beq
	\label{eq:hec_eps_diff}
	\eps_{k+1} - \tilde \eps = c_k(\tilde \eps - \eps_k)^2 + d_k(\eps_{k+1} - \eps_k)^2
\eeq
where
\[
	c_k = \frac{h^{\prime\prime}(\xi_k)}{2h^\prime(\eps_k)} \qquad \text{and} \qquad d_k = - \frac{g_{x_k}^{\prime\prime}(\eta_k)}{2g_{x_k}^\prime(\eps_k)}.
\]
To establish quadratic convergence, we need to bound $\eps_{k+1} - \eps_k$ in terms of~$\tilde \eps - \eps_k$.
To do this, consider the Taylor expansions of $h$ and $g_{x_k}$ but with only the first two terms, \ie
\[
	0 = h(\tilde \eps) = h(\eps_k) + h^\prime(\zeta_k)(\tilde \eps - \eps_k) 
\]
for some $\zeta_k \in  [\tilde \eps, \eps_k]$ and
\[
	0 = g_{x_k}(\hat \eps_k) = g_{x_k}(\eps_k) + g_{x_k}^\prime(\tau_k)(\hat \eps_k - \eps_k) 
\]
for some $\tau_k \in [\tilde \eps, \eps_k]$.
As $h(\eps_k) = g_{x_k}(\eps_k)$ and $\eps_{k+1} = \hat \eps_k$, it follows that 
\begin{equation}
	\label{eq:squared1}
	\frac{\eps_{k+1} - \eps_k}{\tilde \eps - \eps_k} = \frac{h^\prime(\zeta_k)}{g_{x_k}^\prime(\tau_k)},
\end{equation}
which converges to 1 as $k \to \infty$, since $h^\prime(\zeta_k)$ and $g_{x_k}^\prime(\tau_k)$ 
both\footnote{Note that in \cite[p.~1000]{MitO16}, there is a typo: in the second to last line of the proof of Theorem~4.4,
$g_{u_kv_k}^\prime(\eps_k)$ actually should be $g_{u_kv_k}^\prime(\tau_k)$.}
converge to~$h^\prime(\tilde \eps) \neq 0$.
Dividing \eqref{eq:hec_eps_diff} by $(\eps_k - \tilde \eps)^2$ and taking the absolute value yields
\[
	\frac{| \eps_{k+1} - \tilde \eps | }{(\eps_k - \tilde\eps)^2} 
	= \left| c_k+ d_k \left( \frac{\eps_{k+1} - \eps_k}{\eps_k - \tilde \eps} \right)^2 \right|.
\] 
By~\eqref{eq:squared1}, the squared term on the right converges to~$1$ as $k \to \infty$, 
while $c_k$ and $d_k$ also converge since their numerators are bounded and 
their denominators each converge to~$h^\prime(\tilde\eps) \neq 0$.
Thus, \cref{alg:hec} converges Q-quadratically.
\end{proof}

\begin{remark}
\label{rem:hec_roc}
A key part of the proof of~\cref{thm:hec_quad} 
is that the derivatives of~$h$ and~$g_{x_k}$ coincide at $\eps_k$,
which holds because under~\cref{asm:subs}, local maximizers of $g_{x_k}$ are computed exactly.
However, for some applications, it may actually be more efficient to solve the expansion phases inexactly 
at first, which in~\cite[Section~4.3]{MitO16} is called early contraction.
If the expansion phases are solved inexactly, 
but the inexactness goes to zero in the limit,
then HEC still converges at least Q-superlinearly; see~\cite[Sections 3.1 and 3.2]{Mit14} and \cite{DemES82}.
This is useful because when the expansion phases are expensive and require
many iterations of optimization, the early contraction strategy
can significantly reduce the cost of the expansion phases 
while only slightly increasing the total number of HEC iterations.
This can result in significantly faster overall runtimes; see~\cite[Section~8]{MitO16}.
As a final comment, note that the quadratic and superlinear rate of convergence results for HEC discussed 
in this paper hold regardless of how fast the contraction and expansions phases are solved; \eg
if the contraction and expansion phases are solved with linearly convergent methods,
HEC still has at least quadratic convergence (or superlinear if early contraction is used).
\end{remark}

\begin{remark}
If HEC converges to a finite number of cluster points of $\{x_k\}$, 
rather than a unique one as supposed in~\cref{thm:hec_quad},
then it is easy to see that if all the other conditions of~\cref{thm:hec_quad} hold for any subsequence $\{x_{k_i}\}$ converging
to a particular cluster point, then $\{\eps_{k_i}\}$ must converge at least quadratically.
Thus, if these conditions also hold for any subsequence to any of the finitely many cluster points,
then we expect that the rate of convergence of $\{\eps_k\}$ should still be quadratic.
\end{remark}

\section{Continuous-time passive systems} 
\label{sec:cont}
Returning to the optimization of passive systems, we first consider the continuous-time case where
the finite-dimensional state-space model $\M \coloneqq \{A,B,C,D\}$ is minimal and is given by \eqref{statespace_c} 
and its corresponding transfer function~$\T$~\eqref{transfer} is thus proper.
Furthermore, for the remainder of the paper, 
we consider passive transfer functions $\T$ (so $m=p$).
We begin with the theoretical background defining the passivity optimization problem we wish to solve, 
which as we will show, is equivalent to a root-min problem.

\subsection{Passivity of continuous-time proper parametric systems}
\label{sec:cont_mvd}
The material here in this subsection is mostly drawn from \cite{MehV20} but is recalled here in a concise way so that 
we can easily refer to it. We also briefly recall definitions and properties following \cite{Wil72a} 
and refer to the literature for proofs and more details. 

Given $\T$, consider the following rational matrix function of $s \in \C$:
\[ \Phi(s) \coloneqq \T^{\mathsf{H}}(-s) + \T(s), \] which coincides with twice the Hermitian part of $\T(s)$ on the imaginary axis:
\[ \Phi(\imath \omega)=[\T(\imath \omega)]^\mathsf{H} + \T(\imath \omega). \] 

\begin{definition} \label{def:passive}
The continuous-time transfer function $\T$ is
\begin{enumerate}
\item {\em passive} if $\Phi(\imath \omega)\succeq 0 $ for all $\ \omega\in \R\cup\{\infty\}$ 
and $\alpha(A) \leq 0$ with any eigenvalues of $A$ occurring on the imaginary axis being semi-simple and with a transfer function residue that is Hermitian and positive semi-definite,
\item {\em strictly passive} if $\Phi(\imath \omega)\succ 0 $ for all $\ \omega\in \R\cup\{\infty\}$ 
and $\alpha(A) < 0$. 
\end{enumerate} 
\end{definition}
Using the matrix
\begin{equation} \label{prls}
W_\mathrm{c}(X,\M)  \coloneqq  \left[
\begin{array}{cc}
- A^{\mathsf{H}}X-XA & C^{\mathsf{H}} - XB \\
C- B^{\mathsf{H}}X & D^{\mathsf{H}}+D 
\end{array}
\right],
\end{equation}
we have the following necessary and sufficient conditions for passivity of a finite-dimensional continuous-time system in state-space form;
see  \cite{Wil72a}.

\begin{theorem}
\label{thm:xi_char}
Let $\M \coloneqq \{A,B,C,D\}$ be a continuous-time minimal system and let its transfer function~$\T$ thus be proper.
Then $\T$ is (strictly) passive if and only if there exists an $X\in \Hn$ 
such that $X \succ 0$ and $W_\mathrm{c}(X,\M)\succeq 0$ ($W_\mathrm{c}(X,\M)\succ 0$).
\end{theorem}

In \cite{MehV20}, the following class of systems, parameterized by $\xi\in\R$, was considered:
\begin{subequations}
	\label{eq:lti_xi_cont}
	\begin{align}
	\M_\xi \coloneqq {}& \{A_\xi,B,C,D_\xi\} = \{ A+\tfrac{\xi}{2} I_n, B, C, D-\tfrac{\xi}{2} I_m \},\\
	\T_\xi(s) \coloneqq {}& C(sI-A_\xi)^{-1}B+D_\xi 
		= C((s - \tfrac{\xi}{2})I_n - A)^{-1}B+ D - \tfrac{\xi}{2}I_m , \\
	\Phi_\xi(s) \coloneqq {}& \T_\xi^{\mathsf{H}}(-s) + \T_\xi(s).
	\end{align}
\end{subequations}
For perturbations~$\Delta_\M = \{\Delta_A, \Delta_B, \Delta_C, \Delta_D\}$ allowed in the 
system model~$\M$, 
the size of the smallest value of $\|\begin{bsmallmatrix} \Delta_A & \Delta_B \\ \Delta_C & \Delta_D \end{bsmallmatrix}\|\fro$ 
at which the perturbed model~$\M + \Delta_\M$ loses passivity
depends on this scalar parameter~$\xi$ only; see~\cite{MehV20}.
It is therefore important to compute the values of $\xi$
for which these parametric systems are passive or strictly passive.
The following theorem, which is a combination of \cite[Theorem~4.5 and Lemma~6.2]{MehV20},
classifies these values~of~$\xi$.
Note that strict passivity of $\T$ implies regularity of the rational matrix function~$\Phi$, since~$\Phi(\infty)$ is invertible.
Thus, $\Phi_\xi$ is also guaranteed to be regular for almost all $\xi$, since~$\Phi_\xi(\infty)$ is
invertible for almost all $\xi$.

\begin{theorem}
\label{thm:xi_cont}
Let $\M \coloneqq \{A,B,C,D\}$ be a continuous-time minimal system and let its transfer function~$\T$ thus be proper. 
Then, for any $\xi \in \R$, the parametric system $\M_\xi$ with transfer function $\T_\xi$,
as defined in \eqref{eq:lti_xi_cont}, is also minimal and 
\begin{equation} 
	\label{eq:sup_cont}
	\Xi  \coloneqq \sup_{-\infty < \xi < \infty} \left\{\xi : \T_\xi \text{ is strictly passive}\right\}
	= \max_{-\infty < \xi < \infty} \left\{\xi : \T_\xi \text{ is passive}\right\}
\end{equation}
is bounded.  Moreover, 
$\T_\xi$ is strictly passive for $\xi \in (-\infty,\Xi)$, passive but not strictly passive for $\xi = \Xi$, and 
non-passive for $\xi \in (\Xi, +\infty)$.
\end{theorem}
\begin{proof}
It is obvious that adding $\frac{\xi}{2}I$ to the matrix $A$ does not affect controllability or observability since it is a mere shift of the variable $\lambda$ in the corresponding rank conditions. The rest of the proof is based on the fact that passivity of $\T_\xi$ is 
linked to the inequality $W_\mathrm{c}(X,\M)\succeq \xi \diag(X,I_m)$ for some $X\succ 0$, 
 and that strict passivity of $\T_\xi$ is linked to the strict inequality $W_\mathrm{c}(X,\M)\succ \xi \diag(X,I_m)$ for some $X\succ 0$. 
Consequently, for all $\tilde \xi < \xi$, $\T_\xi$ being passive implies that $\T_{\tilde \xi}$ is strictly passive,
and if~$\T_\xi$ is strictly passive, then $\T_{\tilde \xi}$ is also strictly passive in an open neighborhood about~$\xi$.
 This proves that the interval for $\xi$ corresponding to strictly passive systems and non-passive systems are both open and connected, and that there is a single 
 boundary point $\Xi$ that must be passive, but not strictly passive. The boundedness of $\Xi$ follows from the minimality of the realization for $\T_\xi$.
 We refer to \cite{MehV20} for the details.
\end{proof}

By computing $\Xi$,  we can ascertain whether $\M_\xi$ is (strictly) passive or non-passive for all~$\xi \in \R$.
For a given value of~$\xi$, by~\cref{def:passive}, $\M_\xi$ corresponds to a strictly passive system if and only if
$\Phi_\xi(\imath \omega)\succ 0$ for all $\omega \in \R \cup \{\infty\}$,
and $\alpha(A_\xi) < 0$.
Checking asymptotic stability is done via computing $\alpha(A_\xi)$.
When $n$ is not too large, the entire spectrum of $A_\xi$ can be computed to obtain $\alpha(A_\xi)$,
while for large $n$, one can use, \eg \texttt{eigs} in \matlab, to efficiently compute a globally rightmost eigenvalue of $A_\xi$.
Checking the positive definiteness condition is more involved. 
For $\xi,\omega \in \R$, consider 
\beq
	\label{eq:gamma_cont}
	\gamma(\xi,\omega) \coloneqq \lambda_\mathrm{min} (\Phi_\xi(\imath \omega))
	\quad \text{and} \quad
	\begin{aligned}
		\gamma_\xi(\omega) &\coloneqq \gamma(\xi,\omega),  \ \ \text{where $\xi \in \R$ is fixed}, \\
		\gamma_\omega(\xi) &\coloneqq \gamma(\xi,\omega), \ \ \text{where $\omega \in \R$ is fixed}.
	\end{aligned}
\eeq
Clearly, $\Phi_\xi(\imath \omega) \succ 0$ if and only if $\gamma_\xi(\omega) > 0$,
and at $\omega = \infty$, this is simply equivalent to~$D_\xi\ct + D_\xi \succ 0$, with 
$\lim_{\omega \to \infty} \gamma_\xi(\omega) = \lambda_\mathrm{min}(D_\xi\ct + D_\xi) > 0$.
If~$\alpha(A_\xi) < 0$, then~$\T_\xi$ has no poles on the 
imaginary axis, and so neither does $\Phi_\xi$; hence, 
$\gamma_\xi$ is a continuous function.
Thus, if $D_\xi\ct + D_\xi \succ 0$ and $\alpha(A_\xi) < 0$,
then~$\gamma_\xi(\omega_1) \leq 0$ if and only if
$\det \Phi_\xi(\imath\omega_2) = 0$ for~$\omega_1, \omega_2 \in \R$,
with $\omega_1 = \omega_2$ not necessarily holding.
Summarizing, we have the following necessary and sufficient 
algebraic continuous-time conditions for the strict passivity of $\T_\xi$:
\begin{enumerate}
	\item[(C1)] $\alpha(A_\xi) < 0$ (asymptotic stability of $A_\xi = A + \frac{\xi}{2}  I_n$),
	\item[(C2)] $D_\xi\ct +D_\xi = D\ct+D - \xi I_m \succ 0$ (positive definiteness at $\omega=\infty$),
	\item[(C3)] $\det \Phi_\xi(\imath \omega) \neq 0$ for all $\omega \in \R$ 
		(implying positive definiteness for all finite $\omega$ provided that (C1) and (C2) also hold).
\end{enumerate}

A bracket containing $\Xi$ can be easily computed.
A simple lower bound on $\Xi$ is 
\begin{equation}
	\label{eq:xilb_cont}
	\Xilb \coloneqq \lambda_{\min}(W_\mathrm{c}(I_n,\M)),
\end{equation}
as clearly 
\[ 
 	W_\mathrm{c}(I_n,\M_{\Xilb}) = W_\mathrm{c}(I_n,\M) - \Xilb I_{n+m} \succeq 0
\]
holds, and so by~\cref{thm:xi_char,thm:xi_cont}, $\T_{\Xilb}$ is passive.  
Meanwhile, (C1) and (C2) will no longer be satisfied if~$\xi$ is too large:
(C1) holds if and only if $\xi < -2 \alpha(A)$, and 
(C2) holds if and only if $\xi < \lambda_{\min}\left(D^{\mathsf{H}}+D\right)$.
Thus, a simple upper bound for $\Xi$ is
\begin{equation}
	\label{eq:xiub_cont}
 	\Xiub \coloneqq \min \left\{ -2 \alpha(A), \lambda_{\min}\left(D^\mathsf{H}+D\right) \right\}.
\end{equation}

Let us now look at $\xi \in [\Xilb,\Xiub)$, where for this half open interval, $\alpha(A_\xi) < 0$ and $D^\mathsf{H}_\xi +D_\xi \succ 0$. Therefore, in order to verify the strict passivity of $T_\xi$, one only needs to verify condition (C3), 
\ie that $\det \Phi_\xi(\imath \omega) \neq 0$ for all $\omega \in \R$.
This condition can be checked via the following result for $\M_\xi$, 
which is well known in the literature for general systems~$\M$ (see, \eg \cite{Meh91}).

\begin{theorem}
\label{thm:zeros_cont}
Let $\xi \in \R$, $\M_\xi$ and $\Phi_\xi$ be as defined in \eqref{eq:lti_xi_cont}, 
and $\omega \in \C$ (not only $\R$) be any point such that $\imath \omega \not \in \Lambda(A_\xi)$.
Then $\det \Phi_\xi(\imath \omega) = 0$ if and only if $\det (M_\xi - \omega N) = 0$, 
where the regular Hermitian pencil $M_\xi - \lambda N$ is defined by
\begin{equation} 
	\label{eq:eig_extended_cont}
	M_\xi \coloneqq
	\begin{bmatrix} 
		0 & A_\xi & B \\ 
		A_\xi^\mathsf{H} & 0 & C^{\mathsf{H}} \\
		B^\mathsf{H} & C & D_\xi^{\mathsf{H}}+D_\xi 
	\end{bmatrix}
	\ \ \text{and} \ \
	N \coloneqq
	\begin{bmatrix} 
		0 & \imath I_n & 0 \\ 
		-\imath I_n & 0 & 0 \\
		0 & 0 & 0 
	\end{bmatrix}.
\end{equation}
Furthermore, if $D_\xi^\mathsf{H} + D_\xi$ is nonsingular, 
then $\det \Phi_\xi(\imath \omega) = 0$ if and only if $\det (H_\xi - \imath \omega I_{2n}) = 0$, 
where $H_\xi$ is the Hamiltonian matrix 
\begin{equation}
	\label{eq:eig_cont}
	H_\xi \coloneqq
	\begin{bmatrix} A_\xi & 0 \\ 0 & -A_\xi^{\mathsf{H}} \end{bmatrix} 
	- \begin{bmatrix}  B \\ C^{\mathsf{H}} \end{bmatrix} 
	\begin{pmatrix} D_\xi^{\mathsf{H}}+D_\xi \end{pmatrix}^{-1}
	\begin{bmatrix} C & -B^{\mathsf{H}} \end{bmatrix}.
\end{equation}
\end{theorem}
\begin{proof} 
Writing
\[
	M_\xi-\omega N \coloneqq   
	\begin{bmatrix} 
		0 & A_\xi - \imath\omega I_n \! & B \\ 
		 A_\xi^\mathsf{H} + \imath\omega I_n  & 0 & C^{\mathsf{H}} \\
		B^\mathsf{H} & C & D_\xi^{\mathsf{H}}+D_\xi 
	\end{bmatrix},
\]
and using the Schur identity of determinants with respect to the leading $2n\times 2n$ block,
which by assumption is nonsingular,
we obtain that 
\[ \det ( M_\xi - \omega N)=\det\begin{bmatrix} 
		0 & A_\xi -\imath\omega I_n \\ 
		A_\xi^\mathsf{H}+\imath\omega I_n & 0 
	\end{bmatrix}
\det \Phi_\xi(\imath \omega).
\] 
As $\imath \omega \not \in \Lambda(A_\xi)$, 
the first equivalence involving $M_\xi - \lambda N$ holds, and since $\Phi_\xi(\imath \omega)$ is regular,
this matrix pencil must be a regular one.
To obtain the second equivalence, we consider
$(H_\xi-\imath\omega I_{2n})
	\begin{bsmallmatrix} 
		0 & -I_n  \\ I_n  & 0 
	\end{bsmallmatrix}
$,
which is the Schur complement of $M_\xi -\omega N$
with respect to the trailing $m\times m$ block
and requires the additional assumption that $D_\xi^{\mathsf{H}}+D_\xi$ is nonsingular.
Then, via the Schur identity of determinants, we have that
\[ \det (M_\xi - \omega N)=\det (D_\xi^\mathsf{H}+D_\xi) 
\det \left( (H_\xi-\imath\omega I_{2n})\begin{bmatrix} 
		0 & -I_n  \\ I_n  & 0 
	\end{bmatrix}\right).
\] 
\end{proof}

\begin{corollary}
\label{cor:zeros_cont}
Let $\Xilb$ and $\Xiub$ be as defined in~\eqref{eq:xilb_cont} and~\eqref{eq:xiub_cont}, respectively,
and let $\xi \in [\Xilb,\Xiub)$.
Then the function $\gamma_\xi : \R \to \R$ defined in~\eqref{eq:gamma_cont}
has at most $2n$ zeros, all of which must be finite.
\end{corollary}
\begin{proof}
If $\gamma_\xi(\omega)=0$, then $\det \Phi_\xi(\imath \omega) = 0$.
Since $\xi \in [\Xilb,\Xiub)$, (C1) and (C2) both hold, and so
the assumptions of~\cref{thm:zeros_cont} are met.
Hence, $\det \Phi_\xi(\imath \omega) = 0$ if and only if $\det (H_\xi - \imath \omega I_{2n}) = 0$.
Finally, as $H_\xi \in \C^{2n \times 2n}$, it has $2n$ (finite) eigenvalues.
\end{proof}

Given the bracket $[\Xilb,\Xiub]$,
\cref{thm:zeros_cont} immediately leads to a bisection method 
for computing $\Xi$ \cite[p.~144]{MehV20}.  For any $\xi \in [\Xilb,\Xiub)$,
(C3) can be verified by computing the eigenvalues of either $M_\xi - \lambda N$ 
or $H_\xi$ (the pencil form is preferred numerically, since it only has a linear dependence on $\xi$).
Via the following result, Mehrmann and Van Dooren also proposed a second improved algorithm
for computing $\Xi$~\cite[p.~146]{MehV20}.

\begin{theorem} 
\label{thm:root_min_cont}
Let $\gamma : \R \times \R \to \R$ and $\gamma_\xi : \R \to \R$ be as defined in \eqref{eq:gamma_cont}
and~$\Xilb$ and~$\Xiub$ be as defined in~\eqref{eq:xilb_cont} and~\eqref{eq:xiub_cont}, respectively.
Then~$\gamma$ is continuous on the domain $[\Xilb,\Xiub) \times \R$,
and $\gamma_\xi$ has the following properties:
\begin{enumerate}[label=(\roman*),font=\normalfont]
\item if $\xi\in[\Xilb,\Xi)$, then $\min_{\omega \in \R} \gamma_\xi(\omega) > 0$, 
\item if $\xi=\Xi$, then $\min_{\omega \in \R} \gamma_\xi(\omega) = 0$, 
\item if $\xi \in (\Xi,\Xiub)$ with $\Xi < \Xiub$, then $\gamma_\xi(\omega) < 0$ for $\forall \omega \in \mathcal{S} \subset \R$,
	where $\mathcal{S}$ consists of a finite number of non-overlapping open bounded intervals.
\end{enumerate}
\end{theorem}
\begin{proof}
 This was proven in \cite[Theorem~5.1]{MehV20} except for the claim in (iii) that the intervals 
 are bounded, which follows directly from~\cref{cor:zeros_cont}.
\end{proof}

Mehrmann and Van Dooren's improved method computes a monotonically decreasing 
sequence $\{\xi_k\} \to \Xi$, where the initial estimate is $\xi_0 = \Xiub  - \tau$
for some small tolerance~$\tau > 0$. 
On the $k$th iteration, via~\cref{thm:zeros_cont} and computing 
the eigenvalues of $M_{\xi_k} - \lambda N$, the bounded intervals where $\gamma_{\xi_k}$ is negative 
are obtained.  Taking $\hat \omega$ to be the midpoint of the largest of these intervals with $\gamma_{\xi_k}(\hat \omega) < 0$
holding, $\xi_{k+1}$ is obtained by setting it to the smallest value of $\xi$ such that $\gamma_{\hat\omega}(\xi) = 0$;
this is done by computing all the eigenvalues of a matrix pencil closely related to~$M_{\xi_k} - \lambda N$ (and of the same order).
This process is continued in a loop until convergence to $\Xi$; see 
\cite[section~5]{MehV20} for more details.

\subsection{An HEC-based algorithm for computing the continuous-time~$\Xi$}
\label{sec:hec_alg_cont}
By~\cref{thm:root_min_cont}, $\Xi$ defined in \eqref{eq:sup_cont} can 
instead be computed via this root-min problem:
\beq
	\label{eq:root_min_cont}
	\text{Determine the}~\xi \in \Done: \qquad 
	f(\xi) 
	= \min_{\omega \in \Dtwo} \gamma(\xi,\omega) 
	= \min_{\omega \in \Dtwo} \gamma_\xi(\omega) 
	= 0,
\eeq
where $\Done = [\Xilb,\Xiub]$, $\Dtwo \subset \R$, and 
the function $\gamma$ and its associated restriction $\gamma_\xi$ are defined in~\eqref{eq:gamma_cont}.
If~\mbox{$\xi \neq \Xi\ub$}, 
by~\cref{cor:zeros_cont,thm:root_min_cont},
$\Dtwo$ can be taken to be compact, 
since for all other values of $\xi \in \Done$,
$\gamma_\xi$ has at most $2n$ zeros, which are all finite,
and minimizers of~$\gamma_\xi$ that occur 
where $\gamma_\xi$ is negative clearly must lie
between these zeros.  
Since $\alpha(A_\xi) < 0$ for all~$\xi < \Xi\ub$,
it follows that $\gamma_\xi$ is also bounded below (and above) for all~$\xi < \Xiub$.
Thus, \eqref{eq:root_min_cont} meets the criteria to be a valid root-min problem,
and so we can use HEC to find pseudoroots of it.

\begin{remark}
\label{rem:root_min_init}
For a root-min problem, the initialization requirements for HEC, 
using the notation of~\eqref{eq:root_max} and \cref{alg:hec},
are $\eps\lb,\eps_0 \in \Done$ and $x_0 \in \Dtwo$ 
such that 
$f(\eps_0) \leq g(\eps_0,x_0) \leq 0 < f(\eps\lb)$.
For computing $\Xi$, 
it will also be more convenient to use the convention that 
\mbox{$f(\eps_0) \leq g(\eps_0,x_0) < 0 \leq f(\eps\lb)$} holds.
\end{remark}

\begin{remark}
\label{rem:not_zero}
We wish to compute $\Xi$ to a desired relative accuracy determined by 
a relative error tolerance $\tau \in (0,1)$.  
However, if~$\Xi = 0$ happens to hold, the relative error is undefined.
In this case, our algorithm instead automatically estimates $\Xi$ to within an absolute error
determined by $\tau$.  
In practice, one could use two parameters to respectively determine
acceptable accuracy in relative and absolute senses.   
For sake of simplicity, we assume that $\Xi \neq 0$ in order
to avoid having to refer to both relative and absolute errors.
\end{remark}

First note that $0 \leq f(\Xi\lb)$ always holds.
If we have a $\xi_0 \in [\Xilb,\Xiub)$ and \mbox{$\omega_0 \in \R$} 
such that $\gamma(\xi_0,\omega_0) <  0$, 
then the initialization conditions of HEC are met,
and so HEC can be used to compute a 
pseudoroot $(\tilde\xi,\tilde\omega)$ of~\eqref{eq:root_min_cont}
with both~$\tilde \xi \in [\Xilb,\xi_0)$ and~$\Xi \leq \tilde \xi$ holding.
To determine whether the estimate~$\tilde\xi$ is sufficiently
close to~$\Xi$, we do the following.
For our tolerance $\tau$,
we set \mbox{$\xi = \tilde \xi - \tau |\tilde \xi|$} (we assume that~$\Xilb < \xi$, as otherwise we are done)
and then compute the real eigenvalues of~\mbox{$M_\xi - \lambda N$}.
If this matrix pencil has no real eigenvalues, then by~\cref{thm:zeros_cont}, 
we know that $\gamma_\xi$ has no zeros, 
and so by~\cref{thm:root_min_cont}, $\xi < \Xi$ must hold.
Thus,~$\Xi \in (\tilde\xi,\xi]$, and so~$\xi$ must agree 
with~$\Xi$ to the desired number of digits.\footnote{Note that 
when $\tilde \xi$ is only sufficiently close to $\Xi$ (and not equal), 
it does not necessarily follow that $\tilde\omega$ is close to the 
global minimizer(s) of~$\gamma_\Xi$.}
Otherwise, if~$M_\xi - \lambda N$ does have real eigenvalues,
then~$\gamma_\xi$ has zeros, and if~$\gamma_\xi$ is negative
on at least one of the intervals derived from these zeros,
then~$\Xi < \xi$ holds by~\cref{thm:root_min_cont}.
Updating~$\omega_0$ to be the midpoint of one of these intervals
where~$\gamma_\xi$ is negative, obviously~$\gamma_\xi(\omega_0) < 0$ holds,
and so HEC can be restarted to find a new pseudoroot~$(\hat\xi,\hat\omega)$ of~\eqref{eq:root_min_cont}
with \mbox{$\hat\xi \in [\Xi,\xi)$}.
This process of running HEC and computing the real eigenvalues of~\mbox{$M_\xi - \lambda N$}
to find regions where~$\gamma_\xi$ is negative
is repeated in a loop until estimate~$\xi$, which is decreasing monotonically, 
becomes sufficiently close to~$\Xi$.

For initializing our new algorithm, it is required that we always choose $\xi < \Xiub$; \eg
evaluating $\T_\xi(\imath \omega)$ requires solving linear systems with $\imath \omega I - A_\xi$,
but this matrix may not always be invertible when $\xi=\Xiub$.
Choosing $\xi_0 = \Xiub -\tau|\Xiub|$ as our first estimate suffices, 
as this still allows us to obtain $\Xi$ to the desired accuracy. 
Again, we assume that~$\Xilb < \xi_0$, as otherwise we are done.  
The user provides some $\omega_0 \in \R$ as an initial guess
for HEC\@.  
If $\gamma_{\xi_0}(\omega_0) < 0$ holds, then our algorithm
as described above can begin.
Otherwise, we must find another point where~$\gamma_{\xi_0}$ is negative.
This can be done in multiple ways.  We can evaluate~$\gamma_{\xi_0}$ 
on a grid or randomly chosen points.  
We could also initialize some optimization solver at these points
to try to find a minimizer~$\tilde \omega$ of~$\gamma_{\xi_0}$ 
such that $\gamma_{\xi_0}(\tilde \omega) < 0$.
If, after some reasonable amount of effort, such a point
has not been found, we then resort to computing the eigenvalues of~$M_{\xi_0} - \lambda N$
in order to obtain all the zeros of $\gamma_{\xi_0}$.
Then, as described above, we can determine if there exists a point where $\gamma_{\xi_0}$ is negative.
Since evaluating~$\gamma_{\xi_0}(\omega)$ for a given value of $\omega$ is much cheaper than computing the eigenvalues 
of~$M_{\xi_0} - \lambda N$ (we elaborate on this momentarily), it is generally beneficial in terms of the overall runtime to first
try a decent number of points, possibly with optimization.
This also increases the chances that the first pseudoroot~$(\tilde \xi, \tilde \omega)$ of~\eqref{eq:root_min_cont}
found by HEC also provides its root, \ie \mbox{$\tilde \xi = \Xi$}; in this case, our algorithm only computes 
the eigenvalues of $M_\xi - \lambda N$ for a single value of $\xi$.
In contrast, recall that the earlier algorithm of Mehrmann and Van Dooren 
(described at the end of~\cref{sec:cont_mvd}), on every iteration,  
requires computing the eigenvalues of~$M_\xi - \lambda N$ plus 
the eigenvalues of a second related matrix pencil with the same order~\mbox{($2n + m$)}.

\begin{algfloat}[!t]
\begin{algorithm}[H]
\caption{HEC-based algorithm for continuous-time $\Xi$}
\label{alg:xi_cont}
\begin{algorithmic}[1]
	\setcounter{ALC@unique}{0}
	\REQUIRE{ 
		$\M$, $\omega_0 \in \R$, $\tau \in (0,1)$, $\Xilb$ \eqref{eq:xilb_cont}, and $\Xiub$ \eqref{eq:xiub_cont}
	}
	\ENSURE{ 
		$\xi$ such that $| \Xi - \xi | \leq \tau |\Xi|$ for continuous-time $\Xi$ for $\M$
		\\ \quad
	}
	
	\STATE $\xi \gets \Xiub - \tau | \Xiub|$
	\IF { $\xi \leq \Xilb $ }
		\RETURN
	\ENDIF
	\STATE \texttt{find\_negative} $\gets \gamma_\xi(\omega_0) \geq 0 $  \COMMENT {a boolean variable}
	\WHILE { true }
		\IF { \texttt{find\_negative} }
			\STATE $\Omega \gets \{ \omega \in \R : \det ( M_\xi - \omega N ) = 0 \}$
			\STATE $\Omega \gets \{ \omega \in \Omega : \gamma_\xi(\omega) = 0 \}$
			\IF { $\exists \omega_1, \omega_2 \in \Omega$ s.t. $\omega_1 < \omega_2$ and $\gamma_\xi(w) < 0 \ \forall \omega \in (\omega_1,\omega_2)$ }
				\STATE $\omega_0 \gets 0.5(\omega_1 + \omega_2)$  \COMMENT {\emph{(C3)} does not hold}
			\ELSE
				\RETURN  \COMMENT {$\gamma_\xi(\omega) \geq 0 \ \forall\omega \in \R$ 
						and $\xi \approx \Xi$ to tolerance}
			\ENDIF 
		\ENDIF
		\STATE \COMMENT {$\gamma_\xi(\omega_0) < 0$ and $\Xi \in [\Xilb,\xi)$ so run HEC with this initial data}
		\STATE $(\tilde \xi,\tilde \omega) \gets$ a pseudoroot of \eqref{eq:root_min_cont} obtained by HEC with $\Xi \leq \tilde \xi < \xi$
		\STATE $\xi \gets \tilde \xi - \tau |\tilde \xi| $ 
		\STATE \texttt{find\_negative} $\gets$ \texttt{true}
	\ENDWHILE
\end{algorithmic}
\end{algorithm}
\vspace{-0.4cm}
\algnote{Per~\cref{rem:not_zero}, we assume that $\Xi \neq 0$, and so at termination,
$\xi$ will agree with $\Xi$ to the desired relative tolerance~$\tau$.
When the matrices defining $\M$ are all real, there is symmetry, 
\ie $\gamma(\xi,-\omega) = \gamma(\xi,\omega)$, 
and so the search domain for $\omega$ can be reduced from~$\R$ to~\mbox{$\omega \in [0,\infty)$}.
While taking advantage of this symmetry does not affect the asymptotic work complexity,
it can nevertheless reduce the constant factors to speed up the overall run time.
}
\end{algfloat}

Pseudocode of our new algorithm for continuous-time~$\Xi$ is given in~\cref{alg:xi_cont}.  
In practice, we observe that HEC is only restarted a handful of times, often just one.
By construction of the valid data to initialize HEC on every iteration of~\cref{alg:xi_cont}, 
it follows from~\cref{thm:hec_converge} that HEC is indeed
guaranteed to compute a pseudoroot of~\eqref{eq:root_min_cont}
on every iteration.
Under mild assumptions that generally hold in practice,
we now show that~\cref{alg:xi_cont} has local quadratic
convergence to pseudoroots of~\eqref{eq:root_min_cont}.

\begin{theorem}[Quadratic convergence of~\cref{alg:xi_cont}]
\label{thm:cont_quad}
Let $(\tilde \xi, \tilde \omega)$ be any pseudoroot of~\eqref{eq:root_min_cont} computed
by HEC within~\cref{alg:xi_cont}.  If
\begin{enumerate}[label=(\roman*),font=\normalfont]
\item after some point, HEC only generates iterates that lie on a single open continuous path
	$\omega_\mathrm{p} : \R \to \R$ of stationary points of $\gamma_\xi$ as $\xi$ varies,
\item $\tilde \omega = \omega_\mathrm{p}(\tilde \xi)$ is a local maximizer of $\gamma_{\tilde \xi}$,
\item $\gamma_{\tilde \omega}^\prime(\tilde \xi) \neq 0$,
\item $\gamma$ is twice continuously differentiable at $(\tilde \xi,\tilde \omega)$, and 
\item $\omega_\mathrm{p}$ is twice continuously differentiable at $\tilde \xi$,
\end{enumerate}
all hold, then~\cref{alg:xi_cont} converges Q-quadratically to the pseudoroot~$(\tilde \xi, \tilde \omega)$.
Furthermore, if \emph{(i)--(iii)} hold and $\gamma(\tilde \xi, \tilde \omega)$ corresponds to a 
simple eigenvalue of~$\Phi_{\tilde \xi}(\imath \tilde \omega)$ and 
$\gamma_{\tilde \xi}^{\prime\prime}(\tilde \omega) \neq 0$, 
then \emph{(iv)} and~\emph{(v)} are automatically satisfied. 
\end{theorem}
\begin{proof}
Conditions (i)--(v) implying the quadratic convergence of HEC is simply
a translation of~\cref{thm:hec_quad} to the setting of~\eqref{eq:root_min_cont}.
For the second part of the theorem, if $\gamma(\tilde \xi, \tilde \omega)$ corresponds
to a simple eigenvalue, then $\gamma$ is analytic near 
$(\tilde \xi, \tilde \omega)$.  
Defining the function $g(\xi,\omega) \coloneqq \tfrac{\partial}{\partial \omega} \gamma(\xi,\omega)$,
the path~$\omega_\mathrm{p}$ of stationary points of the functions~$\gamma_\xi$ as $\xi$ varies 
can be characterized by the 
equality~$g(\xi,\omega) = 0$ in a neighborhood about the pseudoroot~$(\tilde \xi, \tilde \omega)$.
Since $g$ is analytic, if $\tfrac{\partial g}{\partial \omega}|_{(\tilde \xi, \tilde \omega)} \neq 0$, 
it follows from the implicit function theorem that we can rewrite our characterization as $g(\xi,\omega_\mathrm{p}(\xi)) = 0$, 
where~$\omega_\mathrm{p}$ is analytic near~$\tilde \xi$.
\end{proof}

Under the smoothness assumptions of~\cref{thm:cont_quad},
in a neighborhood of the pseudoroot in question, 
the contraction and expansion phases within HEC can also be solved 
with fast convergence rates.  
Moreover, even if these assumptions do not hold, 
an extension of the analysis of Boyd and Balakrishnan~\cite{BoyB90} shows that
near any minimizer $\omega$, $\gamma_\xi$ is twice continuously differentiable with Lipschitz
second derivative, even if the minimizer is associated with an eigenvalue of~$\Phi_\xi(\imath \omega)$ 
of multiplicity greater than one; for more details, see~\cite{MitO22}.
Thus, the expansion phases can always be solved quickly using secant or Newton's method,
and any use of the early contraction technique discussed in~\cref{rem:hec_roc} 
should be limited, \eg only in initial iterations when one cannot
necessarily expect to be sufficiently close to the fast convergence regime.

For the contraction and expansion phases,
we now describe how to compute the first and second derivatives 
of the functions~$\gamma_\xi$ and $\gamma_\omega$ defined in~\eqref{eq:gamma_cont}.
Given a simple eigenvalue of a Hermitian matrix~$H$ depending on a parameter $t \in \R$,
formulas for the first and second derivatives of 
that eigenvalue can be found in, \eg \cite{Lan64,OveW95}. 
The matrix derivatives~$H^\prime$ and~$H^{\prime\prime}$ appear in these formulas, so we give below the
first and second matrix derivatives of both~$\Tc$ and~$\Phi_\xi(\imath \omega)$ with respect to $\xi$ and $\omega$.
Via applications of standard matrix derivative formulas and the chain rule,
we have that
\begin{subequations}
\label{eq:T_derivs_cont}
\begin{align}
	\dxi \Tc	&= \tfrac{1}{2}( Z_2 - I_m ), &
	\dxi[2] \Tc 	&= \tfrac{1}{2} Z_3,	\\
	\dw \Tc	&= -\imath Z_2,  &
	\dw[2] \Tc	&= -2 Z_3,
\end{align}
\end{subequations}
where~\mbox{$Z_k \coloneqq C((\imath \omega - \tfrac{\xi}{2})I_n - A)^{-k}B$}, 
and so
\begin{subequations}
\label{eq:Phi_derivs_cont}
\begin{align}
	\dxi \Pc	&= \tfrac{1}{2}(Z_2 + Z_2\ct) - I_m, &
	\dxi[2] \Pc 	&= \tfrac{1}{2}(Z_3 + Z_3\ct), \\
	\dw \Pc	&= -\imath Z_2 + \imath Z_2\ct,  &
	\dw[2] \Pc	&= -2(Z_3 + Z_3\ct).
\end{align}
\end{subequations}
Following a technique of Laub~\cite{Lau81}, we can compute 
the Hessenberg form $A=UHU\ct$, where~$U$ is unitary and~$H$ is Hessenberg, 
and then substitute it into~$Z_1$, which yields~$Z_1 = CU((\imath \omega - \tfrac{\xi}{2})I_n - H)^{-1}U\ct B$.
Computing~$U$ and~$H$ is~$\bigO(n^3)$ work but only needs to be done once.  
Thereafter, the inverses appearing in~$Z_k$ can actually be
applied to a vector in just~$\bigO(n^2)$ work since changing the values of~$\xi$ and $\omega$
cannot cause the Hessenberg structure to be lost.
Thus,~$\Pc$ and its matrix derivatives given in~\eqref{eq:Phi_derivs_cont} can 
be obtained in~\mbox{$\bigO(mn^2 + m^2n)$} work.  Using the convention that 
computing the eigenvalues and eigenvectors of a matrix is an atomic operation with cubic complexity,
the total cost to evaluate $\gamma(\xi,\omega)$ and its first and second derivatives with respect to 
$\xi$ and~$\omega$ is~$\bigO(mn^2 + m^2n + m^3)$ work.  

The cost of~\cref{alg:xi_cont} is dominated by 
computing the zeros of~$\gamma_\xi$.
Since HEC generally converges quickly, as do its expansion and contraction phases,
we can consider that the total number of evaluations of the function $\gamma$ is bounded by a constant.
Hence, in~\cref{alg:xi_cont}, HEC does \mbox{$\bigO(mn^2 + m^2n + m^3)$} work.
Meanwhile, finding the zeros of~$\gamma_\xi$ involves computing 
all the eigenvalues of~$M_\xi - \lambda N$, which itself is $\bigO((n+m)^3)$ work.
Thus, for all but the smallest values of $n$, the HEC portion of~\cref{alg:xi_cont}
should only be a fraction of the cost to compute the eigenvalues of~$M_\xi - \lambda N$.

The ``improved" algorithm of Mehrmann and Van Dooren has the same asymptotic work complexity as our method, 
but the hidden constant factor for their algorithm is much larger.
This is partly because on each iteration, their algorithm solves two large 
eigenvalue problems of order~$2n+m$.
However, it also often requires more iterations than~\cref{alg:xi_cont} does.  
While Mehrmann and Van Dooren did not analyze the convergence properties
of their method, our new framework of root-max problems and HEC also shows 
that their method converges at least Q-superlinearly under generic conditions.
To see this, note that on each iteration,
their method computes a single point where $\gamma_\xi$ is negative (as opposed to finding a local minimizer),
but in the limit, these single points do converge to a minimizer of~$\gamma_\xi$ as~$\xi \to \Xi$.
In other words, their algorithm can also been seen as an HEC iteration with a very aggressive early contraction scheme.
Per~\cref{rem:hec_roc}, such an iteration converges at least superlinearly.

\section{Discrete-time passive systems} 
\label{sec:disc}
We now present the discrete-time analogues of the optimization problem 
and our new algorithm given in~\cref{sec:cont}.
We reuse the continuous-time notation from~\cref{sec:cont}.
for the discrete-time setting since the different meanings should be clear from the context, and it 
allows us to generically refer to either setting when needed.

\subsection{Passivity of discrete-time proper parametric systems}
The material in this subsection closely follows \cite{MehV20a}.
For $z \in \C$, we now consider the rational matrix function
\[  
	\Phi(z) \coloneqq \T^{\mathsf{H}}(z^{-1}) + \T(z), 
\] 
which coincides with twice the Hermitian part of $\T$ on the  unit circle:
\[ 
	\Phi(\eul^{\imath \omega})=[\T(\eul^{\imath \omega})]^\mathsf{H} + \T(\eul^{\imath \omega}).
\] 

\begin{definition}
\label{def:pass_disc}
The discrete-time transfer function $\T$ is 
\begin{enumerate}
\item {\em passive} if $\Phi(\eul^{\imath \omega})\succeq 0$ for all $\omega\in (-\pi,\pi]$ and $\rho(A) \leq 1$, 
\ie its eigenvalues are in the closed unit disk, with any occurring on the unit circle being semi-simple and with a transfer function residue that is Hermitian and positive semi-definite,
\item {\em strictly passive}  if $\Phi(\eul^{\imath \omega})\succ 0$ for all $\omega\in (-\pi,\pi]$
and $\rho(A) < 1$.
\end{enumerate}
\end{definition}
The necessary and sufficient conditions for passivity in the discrete-time case (see~\cite{MehV20a}) now make use of the linear matrix function
\begin{equation} \label{prlz}
W_\mathrm{d}(X,\M)  \coloneqq  \left[
\begin{array}{ccc} X & XA & XB \\
A^{\mathsf{H}}X & X & C^{\mathsf{H}} \\
B^{\mathsf{H}}X  & C & D^{\mathsf{H}}+D 
\end{array}
\right].
\end{equation}

\begin{theorem}
\label{thm:char_disc}
Let $\M \coloneqq \{A,B,C,D\}$ be a discrete-time minimal system and let its transfer function~$\T$ thus be proper.
Then  $\T$ is (strictly) passive if and only there exists an~$X\in \Hn$ such that $X \succ 0$
and $W_\mathrm{d}(X,\M)\succeq 0$ ($W_\mathrm{d}(X,\M)\succ 0$).
\end{theorem}

In \cite{MehV20a}, the following class of parametric systems was considered:
\begin{subequations}
	\label{eq:lti_xi_disc}
	\begin{align}
	\M_\xi \coloneqq {}& \{A_\xi,B_\xi,C_\xi,D_\xi\} =
	\left\{ \tfrac{A}{1-\xi}, \tfrac{B}{1-\xi}, \tfrac{C}{1-\xi}, \tfrac{D-\xi I_m}{1-\xi} \right\}, \\
	\T_\xi(z) \coloneqq {}&  C_\xi(zI_n \! -\! A_\xi)^{-1}B_\xi+D_\xi 
		= \tfrac{1}{1 - \xi} \left( C((1\! - \! \xi)zI_n-A)^{-1}B+ D - \xi I_m \right), \\
	\Phi_\xi(z) \coloneqq {}& \T_\xi^{\mathsf{H}}(z^{-1}) + \T_\xi(z).
	\end{align}
\end{subequations}
where $\xi\in(-\infty,1)$ and it is again important to compute for which values of $\xi$ 
these systems are passive or strictly passive.  
Similar to the continuous-time case, 
it was shown in \cite{MehV20a} that the smallest perturbation
$\Delta_\M$ such that the perturbed model~\mbox{$\M + \Delta_\M$} loses passivity
depends on this scalar parameter $\xi$. 
It is therefore important to compute the values of $\xi$
for which these parametric systems are passive or strictly passive.
The following theorem was given 
in~\cite{MehV20a}, in a slightly modified form; we omit its proof as it is similar to 
that of~\cref{thm:xi_cont}.
Note that strict passivity of~$\T$ again implies  regularity of the rational matrix function~$\Phi$, 
since it is invertible for any point on the unit circle.
Moreover,~$\Phi_\xi$ is then also regular for almost all $\xi$ since it is an analytic perturbation of $\Phi$.

\begin{theorem} \label{thm:suboptimal}
Let $\M \coloneqq \{A,B,C,D\}$ be a discrete-time minimal system and let its transfer function~$\T$ thus be proper.
Then, for $\xi \in (-\infty,1)$, the parametric system $\M_\xi$ with transfer function $\T_\xi$,
as defined in \eqref{eq:lti_xi_disc}, is also minimal and
\beq
	\label{eq:sup_disc}
	\Xi \coloneqq \sup_{-\infty < \xi < 1} \left\{\xi : \T_\xi \text{ is strictly passive}\right\}
	= \max_{-\infty < \xi < 1} \left\{\xi : \T_\xi \text{ is passive}\right\}
\eeq
is bounded.
Moreover, $\T_\xi$ is strictly passive for $\xi \in (-\infty,\Xi)$, passive but not strictly passive for $\xi = \Xi$, and not passive 
for $\xi \in (\Xi, 1)$.
\end{theorem}

By~\cref{def:pass_disc}, $\M_\xi$ is strictly passive if and only if
$\Phi_\xi(z) \succ 0$ holds over the entire unit circle and $\rho(A_\xi) < 1$.
Obtaining the value of $\rho(A_\xi)$ to check asymptotic stability 
can be done by computing an outermost eigenvalue of $A_\xi$ via, 
\eg \texttt{eig} or \texttt{eigs} in \matlab.
Checking the discrete-time positive definiteness condition is a little more subtle
than it is in the continuous-time case.
For $\xi,\omega \in \R$, consider 
\beq
	\label{eq:gamma_disc}
	\gamma(\xi,\omega) \coloneqq \lambda_\mathrm{min} (\Phi_\xi(\eiw)) 
	\quad \text{and} \quad
	\begin{aligned}
		\gamma_\xi(\omega) &\coloneqq \gamma(\xi,\omega)  \ \ \text{where $\xi \in \R$ is fixed}, \\
		\gamma_\omega(\xi) &\coloneqq \gamma(\xi,\omega)  \ \ \text{where $\omega \in \R$ is fixed},
	\end{aligned}
\eeq
where $\Phi_\xi$ is defined in \eqref{eq:lti_xi_disc}.
Clearly $\Phi_\xi(\eiw) \succ 0$ is equivalent to \mbox{$\gamma_\xi(\omega) > 0$},
and~$\gamma_\xi$ is continuous if $\rho(A_\xi) < 1$,
as then $\Phi_\xi$ cannot have any poles on the unit circle.
Hence, if~$\rho(A_\xi) < 1$ and 
$\Phi_\xi(\eul^{\imath \tilde \omega}) \succ 0$ for some \mbox{$\tilde \omega \in (-\pi,\pi]$}, 
then $\Phi_\xi(\eiw) \succ 0$ for all~$\omega \in (-\pi,\pi]$
if and only if $\det \Phi_\xi(\eiw)$ has no zeros.
Thus, $\T_\xi$ is strictly passive if and only if 
the following conditions all hold:
\begin{enumerate}
	\item[(D1)] $\rho(A_\xi) < 1$ (asymptotic stability of $A_\xi = \tfrac{A}{1 - \xi}$),
      	\item[(D2)] $\Phi_\xi(\eul^{\imath \tilde\omega}) \succ 0$ (positive definiteness at a unimodular point, say, $\eul^{\imath \tilde\omega} = 1$),
	\item[(D3)] $\det \Phi_\xi(\eiw) \neq 0$ for all $\omega \in (-\pi,\pi]$ 
		(implying positive definiteness on the entire unit circle provided that (D1) and (D2) also hold).
\end{enumerate}
In contrast to its continuous-time analogue (C2), note that (D2) does not require 
that \mbox{$D_\xi\ct + D_\xi$} be positive definite (or even invertible).

A bracket containing the discrete-time $\Xi$ is as follows.
Again using the relation between the linear matrix inequalities of $\M$ and $\M_\xi$, with $X= 2I_n$,
we can choose 
\begin{equation}
	\label{eq:xilb_disc}
	\Xilb \coloneqq \tfrac{1}{2} \lambda_{\min}W_\mathrm{d}(2I_n,\M)
\end{equation}	
as a lower bound on $\Xi$, 
since it follows that 
\[ 
	(1-\Xilb)W_\mathrm{d}(2I_n,\M_{\Xilb})=W_\mathrm{d}(2I_n,\M)- 2 \Xilb I_{2n+m} \succeq 0,
\]
holds and so by~\cref{thm:char_disc,thm:suboptimal}, we have that $\T_{\Xilb}$ is passive. 
Meanwhile,
\begin{equation}
	\label{eq:xiub_disc}
	\Xiub \coloneqq 1-\rho(A)
\end{equation}
is an upper bound, since obviously $\rho(A_\xi) < 1$ if $\xi < \Xiub$.  

Given $\xi \in [\Xilb,\Xiub)$, (D1) must always hold, so to verify strict passivity of $T_\xi$ 
we need to check that (D2) and (D3) also both hold.
Checking (D2) is simple. 
If~$\lambda_\mathrm{min}(\Phi_\xi(\eul^{\imath \tilde \omega})) \leq 0$ for any $\tilde \omega \in \R$, 
then $T_\xi$ is not strictly passive,
and there is no need to check (D3).
Otherwise, since $\rho(A_\xi) < 1$, if $\lambda_\mathrm{min}(\Phi_\xi(\eul^{\imath \tilde \omega})) > 0$,
we have that $\T_\xi$ is strictly passive if and only if (D3) holds, 
which can be checked via the following result\footnote{The 
generalized eigenvalue problem given by the matrices in~\eqref{eq:eig_extended_disc}
is denoted $\varGamma(\xi,\omega)$ in~\cite[p.~1263]{MehV20a}, but note that
its bottom right block, $D\ct + D - \xi I_m$, contains a typo; 
it should be \mbox{$D\ct + D - 2\xi I_m$}, which we denote $\widetilde D_\xi$ in~\cref{thm:zeros_disc}.
}
(see, \eg \cite{Fas02,Xu06}).

\begin{theorem}
\label{thm:zeros_disc}
Let $\xi \in (-\infty,1)$, $\M_\xi$ and $\Phi_\xi$ be as defined in \eqref{eq:lti_xi_disc},
and $z \in \C$ be any nonzero point such that $z \not \in \Lambda(A_\xi)$
and  $z^{-1} \not \in \Lambda(A_\xi\ct)$, which is equivalent to the former when $|z| = 1$.
Then $\det \Phi_\xi(z) = 0$ if and only if $\det (M_\xi - z N_\xi) = 0$,
where~$\widetilde D_\xi \coloneqq D\ct + D - 2\xi I_m$ and the regular pencil
$M_\xi - \lambda N_\xi$ is defined by
\begin{equation} 
	\label{eq:eig_extended_disc}
	M_\xi \coloneqq
	\begin{bmatrix} 
		0 & A & B \\ 
		(\xi - 1) I_n & 0 & 0 \\
		B^\mathsf{H} & C & \widetilde D_\xi
	\end{bmatrix}
	\ \ \text{and} \ \
 	N_\xi \coloneqq
	\begin{bmatrix} 
		0 & (1 - \xi) I_n & 0 \\ 
		-A\ct & 0 & -C\ct \\
		0 & 0 & 0 
	\end{bmatrix}.
\end{equation}
Furthermore,  if $\widetilde D_\xi$ is nonsingular, 
then $\det \Phi_\xi(z) = 0$ if and only if $\det (S_\xi - z T_\xi) = 0$,
where the symplectic pencil $S_\xi - \lambda T_\xi$ is defined by
\begin{equation} 
	\label{eq:eig_disc}
	S_\xi \coloneqq
	\begin{bmatrix} 
		(\xi - 1) I_n & 0 \\
		-B\widetilde D_\xi^{-1} B\ct & A - B \widetilde D_\xi^{-1} C 
	\end{bmatrix}
	\ \ \text{and} \ \
 	T_\xi \coloneqq
	\begin{bmatrix} 
		(B \widetilde D_\xi^{-1} C - A)\ct 	& C\ct \widetilde D_\xi^{-1} C \\
		0 							& (1 - \xi) I_n \\
	\end{bmatrix}.
\end{equation}
\end{theorem}
\begin{proof} 
Writing
\[
	M_\xi - z N_\xi \coloneqq   
	\begin{bmatrix} 
		0 & A + (\xi - 1)z I_n & B \\ 
		 zA\ct + (\xi - 1)I_n  & 0 & zC\ct \\
		B^\mathsf{H} & C &  \widetilde D_\xi
	\end{bmatrix},
\]
and using the Schur identity of determinants with respect to the leading $2n\times 2n$ block,
which by assumption is nonsingular,
we obtain that 
\[ \det ( M_\xi - z N_\xi)=
	\det
	\begin{bmatrix} 
		0 & A + (\xi - 1)z I_n \\ 
		zA\ct + (\xi - 1)I_n  & 0 
	\end{bmatrix}
	\det ( (1 - \xi) \Phi_\xi(z)).
\] 
As $z \not \in \Lambda(A_\xi)$ and $z^{-1} \not \in \Lambda(A_\xi\ct)$,
the first equivalence involving $M_\xi - \lambda N_\xi$ holds, and since~$\Phi_\xi$ is regular,
this matrix pencil must be a regular one.
To obtain the second equivalence, 
we again apply the Schur identity of determinants, now 
with respect to the trailing $m\times m$ block, which is possible by our additional assumption that $\widetilde D_\xi$ is nonsingular.
It then follows that $\det ( M_\xi - z N_\xi)$ is equal to 
\[
	\det \widetilde D_\xi 
	 \det \left(
	\begin{bmatrix} 
		0 & A + (\xi - 1)z I_n \\ 
		zA\ct + (\xi - 1)I_n  & 0 
	\end{bmatrix}
	- 
	\begin{bmatrix} 
		B \\ z C\ct
	\end{bmatrix}
	\widetilde D_\xi^{-1}
	\begin{bmatrix} 
		B\ct & C
	\end{bmatrix}
	\right),
\]
and so clearly $\det ( M_\xi - z N_\xi) = 0$ if and only if the second determinant above is zero.
Multiplying the matrix inside this second determinant by $\begin{bsmallmatrix} 0 & I_n \\ I_n & 0 \end{bsmallmatrix}$ from the left 
and rearranging terms yields $S_\xi - z T_\xi$.  This matrix pencil 
is easily verified as symplectic, i.e., 
for \mbox{$J \coloneqq \begin{bsmallmatrix} 0 & I_n \\ -I_n & 0 \end{bsmallmatrix}$},
$S_\xi\ct J S_\xi = T_\xi\ct J T_\xi$ holds.
\end{proof}

\begin{corollary}
\label{cor:zeros_disc}
Let $\Xilb$ and $\Xiub$ be as defined in~\eqref{eq:xilb_disc} and~\eqref{eq:xiub_disc}, respectively,
and let $\xi \in [\Xilb,\Xiub)$.
Then the function $\gamma_\xi : (-\pi,\pi] \to \R$ defined in~\eqref{eq:gamma_disc}
has at most~$2n$  zeros.
\end{corollary}
\begin{proof}
If $\gamma_\xi(\omega) = 0$, then $\det \Phi_\xi(\eiw) = 0$.
As $\xi \in [\Xilb,\Xiub)$, \mbox{$\rho(A_\xi) < 1$} holds, and so 
 the assumptions of~\cref{thm:zeros_disc} are met.
Hence, $\det \Phi_\xi(\eiw) = 0$ if and only if $\det (M_\xi - \eiw N_\xi)$.
The proof is completed by noting that $\rank N_\xi \leq 2n$.
\end{proof}

Using~\cref{thm:zeros_disc}, Mehrmann and Van Dooren 
proposed a bisection method to compute discrete-time $\Xi$,
and via the following result, 
a discrete-time analogue of their improved procedure we described in~\cref{sec:cont_mvd}; 
for more details, see~\cite[section~7]{MehV20a}.

\begin{theorem} 
\label{thm:root_min_disc}
Let $\gamma : \R \times (-\pi,\pi] \to \R$ and $\gamma_\xi : (-\pi,\pi] \to \R$ be as defined in~\eqref{eq:gamma_disc},
and~$\Xilb$ and~$\Xiub$ be as defined in~\eqref{eq:xilb_disc} and~\eqref{eq:xiub_disc}, respectively.
Then $\gamma$ is continuous on the domain $[\Xilb,\Xiub) \times (-\pi,\pi]$
and  $\gamma_\xi$ has the following properties:
\begin{enumerate}[label=(\roman*),font=\normalfont]
\item if $\xi\in[\Xilb,\Xi)$, then $\min_{\omega \in (-\pi,\pi]} \gamma_\xi(\omega) > 0$, 
\item if $\xi=\Xi$, then $\min_{\omega \in (-\pi,\pi]} \gamma_\xi(\omega) = 0$ 
\item if $\xi \in (\Xi,\Xiub)$ with $\Xi < \Xiub$, then $\gamma_\xi(\omega) < 0$ $\forall \omega \in \mathcal{S} \subseteq (-\pi,\pi]$,
	where $\mathcal{S}$ consists of a finite number of non-overlapping open intervals.
\end{enumerate}
\end{theorem}
\begin{proof}
Statements (i) and (ii) follow from \cite{MehV20a}, while
(iii) follows from the facts that~$\gamma_\xi$ is continuous,
and by~\cref{cor:zeros_disc}, it can have at most $2n$ zeros.
\end{proof}

\begin{remark}
\label{rem:disc_diff}
For any $\xi\in(\Xi, \Xiub)$, 
the continuous-time function $\gamma_\xi$ always has at least two
zero-crossings,
but note that the discrete-time version of~$\gamma_\xi$ may not have any zeros; \ie
$\max_{\omega \in (-\pi,\pi]} \gamma_\xi(\omega) < 0$ can hold.  
This is why it is necessary to check that both (D2) and (D3) hold at each estimate $\xi$ encountered
when computing discrete-time $\Xi$, but in the continuous-time case, 
only (C3) needs to be checked at each estimate.
Note that the descriptions of the discrete-time algorithms in~\cite[section~7]{MehV20a} do 
not make this important distinction clear.
\end{remark}

\subsection{An HEC-based algorithm for computing discrete-time $\Xi$}
\label{sec:hec_alg_disc}
By~\cref{thm:root_min_disc}, $\Xi$ defined in~\eqref{eq:sup_disc} can be computed via this root-min problem:
\beq
	\label{eq:root_min_disc}
	\text{Determine the}~\xi \in \Done: \qquad 
	f(\xi) = \min_{\omega \in \Dtwo} \gamma(\xi,\omega) = 0,
\eeq
where now $\xi \in \Done = [\Xilb,\Xiub]$, $\Dtwo = (-\pi,\pi]$ is obviously compact, 
$\gamma$ is defined in~\eqref{eq:gamma_disc}, and $\gamma_\xi$ is bounded below.
Our continuous-time HEC-based algorithm and results from~\cref{sec:hec_alg_cont}
extend to the discrete-time setting and work similarly, 
so for brevity, we only focus on the key points and differences.

\begin{algfloat}[t]
\begin{algorithm}[H]
\caption{HEC-based algorithm for discrete-time $\Xi$}
\label{alg:xi_disc}
\begin{algorithmic}[1]
	\setcounter{ALC@unique}{0}
	\REQUIRE{ 
		$\M$, $\omega_0 \in (-\pi,\pi]$, $\tau \in (0,1)$, $\Xilb$ \eqref{eq:xilb_disc}, and $\Xiub$ \eqref{eq:xiub_disc}
	}
	\ENSURE{ 
		$\xi$ such that $| \Xi - \xi | \leq \tau |\Xi|$ for discrete-time $\Xi$ for $\M$
		\\ \quad
	}
	
	\STATE $\xi \gets \Xiub - \tau |\Xiub|$ 
	\IF { $\xi \leq \Xilb $ }
		\RETURN
	\ENDIF
	\STATE \texttt{find\_negative} $\gets \gamma_\xi(\omega_0) \geq 0 $  \COMMENT {a boolean variable}
	\WHILE { true }
		\IF { \texttt{find\_negative} }
			\IF { $\gamma_\xi(0) < 0$ }
				\label{line:xid_d2}
				\STATE $\omega_0 \gets 0$ \COMMENT {\emph{(D2)} does not hold}
				\label{line:xid_change_omega}
			\ELSE
				\STATE $\Omega \gets \{ \omega \in (-\pi,\pi] : \det ( M_\xi - \eiw N_\xi ) = 0 \}$
				\STATE $\Omega \gets \{ \omega \in \Omega : \gamma_\xi(\omega) = 0 \}$
				\STATE $\Omega \gets \Omega \cup \{\min \Omega + 2\pi\}$
				\label{line:xid_wraparound}
				\IF { $\exists \omega_1, \omega_2 \in \Omega$ s.t. $\omega_1 < \omega_2$ and $\gamma_\xi(w) < 0 \ \forall \omega \in 				(\omega_1,\omega_2)$ }
				\STATE $\omega_0 \gets 0.5(\omega_1 + \omega_2)$ \COMMENT {\emph{(D2)} and \emph{(D3)} do not hold}
				\ELSE
				\RETURN  \COMMENT {$\gamma_\xi(\omega) \geq 0 \ \forall\omega \in (-\pi,\pi]$ 
						and $\xi \approx \Xi$ to tolerance}
				\ENDIF 
			\ENDIF
		\ENDIF
		\STATE \COMMENT {$\gamma_\xi(\omega_0) < 0$ and $\Xi \in [\Xilb,\xi)$ so run HEC with this initial data}
		\STATE $(\tilde \xi,\tilde \omega) \gets$ a pseudoroot of \eqref{eq:root_min_disc} obtained by HEC with $\Xi \leq \tilde \xi < \xi$
		\STATE $\xi \gets \tilde \xi - \tau |\tilde \xi| $ 
		\STATE \texttt{find\_negative} $\gets$ \texttt{true}
	\ENDWHILE
\end{algorithmic}
\end{algorithm}
\vspace{-0.4cm}
\algnote{See~\cref{rem:not_zero,alg:xi_cont} for more details
on tolerances and symmetry.
In~\cref{line:xid_wraparound}, $\Omega \cup \{\min \Omega + 2\pi \}$ is used so that the 
``wrap-around" interval, \ie $[\max \Omega, \min \Omega + 2\pi]$ is not missed.
}
\end{algfloat}

Pseudocode for our new algorithm for discrete-time $\Xi$ is given in~\cref{alg:xi_disc}.
Per~\cref{rem:disc_diff}, the need to check that both (D2) and (D3) hold on each iteration
means that the pseudocode is a bit more complicated than 
for continuous-time $\Xi$.  As such, one might conclude that the problem of computing $\Xi$ 
is trickier in the discrete-time case; however, as we explain in the numerical results, 
it seems that the exact opposite is true, due to a numerical issue that only arises in the continuous-time case.
To implement HEC for~\cref{alg:xi_disc},
we make use of the first and second derivatives of~$\gamma_\xi$ and $\gamma_\omega$ defined in~\eqref{eq:gamma_disc}.
To that end, we provide the discrete-time analogues of
the matrix derivatives given in~\eqref{eq:T_derivs_cont}, as the remaining computations are 
readily apparent.  
Letting $Z_k \coloneqq C((1 - \xi)\eul^{\imath \omega}I_n - A)^{-k}B$, we have that
\begin{subequations}
\begin{align}
	\dxi \Td	&= \tfrac{\Td + \eiw  Z_2 - I_m} {1 - \xi}, & 
	\dxi[2] \Td	&= \tfrac{2}{1 - \xi} \left(\eiww Z_3 + \dxiTd \right), \\
	\dw \Td	&= -\imath \eiw Z_2,  &
	\dw[2] \Td	&= \eiw Z_2  - 2(1 - \xi) \eiww Z_3. 
\end{align}
\end{subequations}
In~\cref{alg:xi_disc}, the costs to run HEC and compute zeros of~$\gamma_\xi$ 
are the same as in the continuous-time setting discussed in~\cref{sec:hec_alg_cont}.
\cref{thm:cont_quad} also extends, and so
 under mild assumptions that generally hold in practice, 
\cref{alg:xi_disc} converges quadratically to pseudoroots
of~\eqref{eq:root_min_disc}.
Relatedly, Mehrmann and Van Dooren's improved algorithm~\cite[section~7]{MehV20a}
for discrete-time $\Xi$ also converges at least superlinearly.

\section{Numerical experiments}
\label{sec:experiments}
We implemented the continuous- and discrete-time versions of our new HEC-based method
and the improved midpoint-based iteration of Mehrmann and Van Dooren.  
In this section, for brevity, we use HEC to refer to former (\cref{alg:xi_cont,alg:xi_disc}) 
and MP (for midpoint) to refer to the latter.  All codes were implemented with relative tolerances and set to compute~$\Xi$
to 14~digits.
Experiments were done using \matlab\ R2021a on a 2020 MacBook Pro with a
quad-core Intel Core i5 1038NG7 CPU and 16~GB of RAM running macOS 10.15.7.
Code and data to reproduce all experiments is included in the supplementary material.

\subsection{Implementation details}
We first discuss implementing~\cref{alg:hec}.  
The expansion phase was implemented using \texttt{fmincon}, 
while the contraction phase was implemented using a our own 
Halley-bisection root-finding code; first and second derivative information is used in both.
Due to rounding errors, it may be that contraction phase sometimes 
computes an approximate root $\hat \eps_k$ of~$g_{x_k}$ such that $g_{x_k}(\hat \eps_k) < 0$,
instead of~$g_{x_k}(\hat \eps_k) \geq 0$, which is required at every iteration (for a root-max problem).
However, if this occurs, it suffices to just perturb the computed root by a small multiple 
of the Halley step to correct the sign; a more complicated workaround
involving shifting the root problems is suggested in~\cite[section~7]{MitO16} and 
\cite[Appendix~A]{GugGMetal17}, but we do not recommend that.
\cref{alg:hec} is terminated at an approximate pseudoroot once both $\eps_k$ and $x_k$ 
are no longer changing significantly with respect to their respective previous values; 
this condition is checked twice per iteration, after the contraction phase and after the expansion phase.
Since in the context of computing~$\Xi$, the expansion phases can be solved quickly, we 
did not use early contraction.

For simplicity, we used \texttt{eig} for all eigenvalue problems, 
though it is advisable to use structure-preserving solvers for numerical robustness; 
\eg see~\cite{BenBMX99b,BenBMetal02,KreSW09,Xu06}.
To compute zeros of~$\gamma_\xi(\omega)$, we used the pencils given by the matrices 
in~\eqref{eq:eig_extended_cont} and \eqref{eq:eig_extended_disc} and respectively
identified their real and unimodular eigenvalues using a tolerance.\footnote{If 
$A$, $B$, $C$, and $D$ are all real, then \texttt{eig} returns real 
eigenvalues of~\eqref{eq:eig_extended_cont}
without any rounding error in their imaginary parts; otherwise,
zero imaginary parts may be nonzero numerically.}
Note that if $\gamma_\xi$ has a minimizer or maximizer $\hat \omega$ such
that $\gamma_\xi(\hat \omega) = 0$ (or approximately equal), then this corresponds
to a (nearly) multiple eigenvalue (with multiplicity at least two) 
of the pencil given by~\eqref{eq:eig_extended_cont} or \eqref{eq:eig_extended_disc}, as appropriate.
This always happens as any of the methods approach $\Xi$, and it is generally also true at computed pseudoroots 
and at~$\omega =0$ when the problems have symmetry.
Due to rounding errors, such eigenvalues, even when computed via a structure-preserving solver, 
may not be detected as (close to) real or unimodular.  If this happens, a zero of $\gamma_\xi$ will be missed,
which in turn can cause any of the algorithms to stagnate.  
Fortunately, a robust fix is easy: if $(\tilde \xi, \tilde \omega)$ is the most recent computed pseudoroot, 
simply explicitly add $\tilde \omega$ as a zero of~$\gamma_{\xi}$; 
a similar fix is also necessary for MP.  For more details, see~\cite[pp.~371--373]{BurLO03},
where this fix was proposed in the context of computing the pseudospectral abscissa.

For continuous-time $\Xi$,
there is an additional numerical difficulty when computing the zeros of $\gamma_\xi$  when $\xi \approx \Xiub$.
Although these zeros must be finite, they still may be arbitrarily far away from the origin,
and so there may be large errors in the imaginary parts of 
the corresponding computed real eigenvalues of $M_\xi - \lambda N$.
Mehrmann and Van Dooren recommended using a tolerance so that the first estimate~$\xi$ 
tested was sufficiently far away from $\Xiub$ to help avoid such problems.  
However, we have observed that even a relatively large perturbation 
may still be insufficient to avoid failure of MP.
Our MP code implementing their method uses \mbox{$\xi_0 = \Xiub - |\Xiub| 10^{-4}$}, but 
only small perturbations are done for subsequent estimates in order to obtain the desired 14-digit accuracy;
of course, if~$\Xi \approx \Xiub$, high accuracy may not be possible with MP.
In contrast, our HEC-based method is much less susceptible to this issue, 
since even if only one root of~$\gamma_\xi$ is detected, 
it generally can still be used to start~\cref{alg:hec}.  
Even if this root is a stationary point, a small perturbation to the left or right 
generally yields a point for starting~\cref{alg:hec}.  
In general, structure-preserving eigensolvers can be used 
or one can increase the allowed amount of rounding error in the imaginary part of an eigenvalue
in proportion with the magnitude of the eigenvalue.
 
Finally, in~\cref{line:xid_d2} of~\cref{alg:xi_disc} when checking (D2),
instead of always looking at the sign of~$\gamma_\xi(0)$, 
after the first pseudoroot has been computed
we instead test if~\mbox{$\gamma_\xi(\tilde \omega + \tfrac{1}{2}\pi) < 0$}, and if so
set $\omega_0 \gets \tilde \omega + \tfrac{1}{2}\pi$ in~\cref{line:xid_change_omega}.
The reason is because if the previous pseudoroot has $\tilde \omega = 0$, 
$\gamma_\xi(\tilde \omega) < 0$ almost always holds due to rounding error even though 
it should be exactly zero.  Shifting by, \eg $\tfrac{1}{2}\pi$, ensures that (D2) is checked at a new point;
note that shifting by $\pi$ or $2\pi$ would not ensure this.

\subsection{Experiments}
\begin{table}[t]
\centering
\small
\caption{MP and HEC compared on two continuous-time (cont.) and one discrete-time (disc.)
examples.  For Random, MP was tested in two configurations: MP-fail with~\mbox{$\xi_0 = \Xiub - |\Xiub|10^{-5}$}
and MP with \mbox{$\xi_0 = \Xiub - |\Xiub|10^{-4}$}.
The number of iterations is shown in the ``iters." column; the average number of 
iterations of~\cref{alg:hec} is also given in parentheses for HEC.
The number of eigenvalue problems solved is shown under the \mbox{``\# \texttt{eig} (order, type)"} columns,
separated into the number of order $2n + m$ matrix pencils \mbox{``($2n + m$, P)''}
and the number of order $m$ matrices \mbox{``($m$, M)"}.
The overall running time in seconds is given under ``time (sec.)", while the computed 
estimates for $\Xi$ is given in the rightmost column.
}
\begin{tabular}{l | c | SS | S  | c} 
\toprule
\multicolumn{1}{c}{} & 
\multicolumn{1}{c}{} &
\multicolumn{2}{c}{\# \texttt{eig} (order, type)} &
\multicolumn{1}{c}{} &
\multicolumn{1}{c}{} \\
\cmidrule(lr){3-4}
\multicolumn{1}{c}{Alg.} & 
\multicolumn{1}{c}{iters.} & 
\multicolumn{1}{c}{($2n+m$, P)} & 
\multicolumn{1}{c}{($m$, M)} & 
\multicolumn{1}{c}{time (sec.)} & 
\multicolumn{1}{c}{$\Xi$ estimate} \\
\midrule
\multicolumn{6}{c}{Random ($n = 200$, $m = 10$, cont.) --- $[\Xilb, \Xiub] = [-25.56407,-10.90965]$} \\ 
\midrule
MP-fail      &        1 &  1 &   1 &  0.440 & $-10.9097612001839$ \\ 
MP           &       14 & 27 &  49 &  8.892 & $-14.4073741346323$ \\ 
HEC          &   2(5.0) &  2 &  78 &  0.909 & $-14.4073741346323$ \\ 
\midrule
\multicolumn{6}{c}{RLC    ($n = 200$, $m = 1$, cont.) --- $[\Xilb, \Xiub] = [-32.1267,2.022606]$} \\ 
\midrule
MP           &        4 &  8 &  18 &  2.360 & $0.562483988863916$ \\ 
HEC          &   1(4.0) &  2 &  41 &  0.767 & $0.562483988863891$ \\ 
\midrule
\multicolumn{6}{c}{ISS    ($n = 228$, $m = 3$, disc.) --- $[\Xilb, \Xiub] = [-3.007437,3.117278 \times 10^{-6}]$} \\ 
\midrule
MP           &       15 & 29 & 550 &  8.490 & $-9.37320364701040 \times 10^{-5}$ \\ 
HEC          &   2(4.0) &  1 &  97 &  0.374 & $-9.37320364699013 \times 10^{-5}$ \\ 
\bottomrule
\end{tabular} 
\label{tbl:comp}
\end{table}
We begin with a randomly generated continuous-time example with complex matrices (denoted Random) to illustrate (i) 
when our method encounters at least two pseudoroots before converging (see \cref{fig:hec_rand}) and 
(ii) the aforementioned difficulty of computing 
zeros of~$\gamma_\xi$ when~\mbox{$\xi \approx \Xiub$} (see \cref{fig:mp_fail}).  
In \cref{tbl:comp}, we see that MP is about ten times slower than HEC.
Although HEC required more computations of $\lambda_\mathrm{min} (\Phi_\xi(\eiw))$,
it only needed to solve two of the large eigenvalue problems involving $M_\xi - \lambda N$.
Meanwhile, MP required 27 solves with the pencils and took 14 iterations to converge.
HEC converged to~$\Xi$ at its second pseudoroot, and \cref{alg:hec} on
average took 5.0 iterations to converge to a pseudoroot.

Our second continuous-time example is the electric RLC circuit model used in~\cite{morBenGV20}.
We refer to \cref{fig:hec_rlc} and \cref{tbl:comp} for the complete performance details,
but note that HEC was over three times faster than MP for this RLC example, 
with both methods converging faster and with less work than on the random example.

To compare the discrete-time methods, we used the ISS model from the 
SLICOT benchmark examples.\footnote{Available at \url{http://slicot.org/20-site/126-benchmark-examples-for-model-reduction}.}  Since ISS is a continuous-time model, we converted it to a minimal discrete-time one
by calling \texttt{c2d} using a sampling time of $0.001$ followed by \texttt{minreal}.
In \cref{fig:hec_rlc} and \cref{tbl:comp}, we see that HEC was almost 23 times faster than MP,
again due to the great disparity in the number of large generalized eigenvalue problems solved.
In fact, for ISS, HEC also solved far fewer smaller standard eigenvalues problems as well.
From \cref{tbl:comp} and \cref{fig:mp_iss}, we also see that MP did not quite compute 
$\Xi$ to the requested 14-digit accuracy, while HEC apparently did.  
This slight inaccuracy is the result of MP solving root problems via solving eigenvalue problems,
but such errors can be larger; see the caption of~\cref{fig:mp_iss} for more details.

\begin{figure}[t]
	\begin{subfigure}{\textwidth}
		\centering
		\includegraphics[scale=\threescale,trim=0mm 0mm 0mm 0mm,clip]{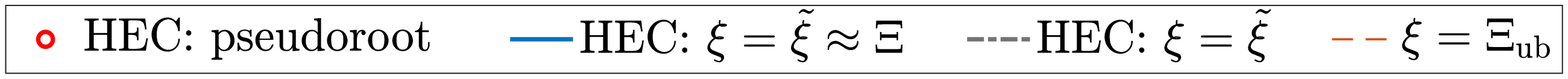}
	\end{subfigure} \\
	\begin{subfigure}{.329\textwidth}
		\centering
		\includegraphics[scale=\threescale,trim=0mm 0mm 0mm 0mm,clip]{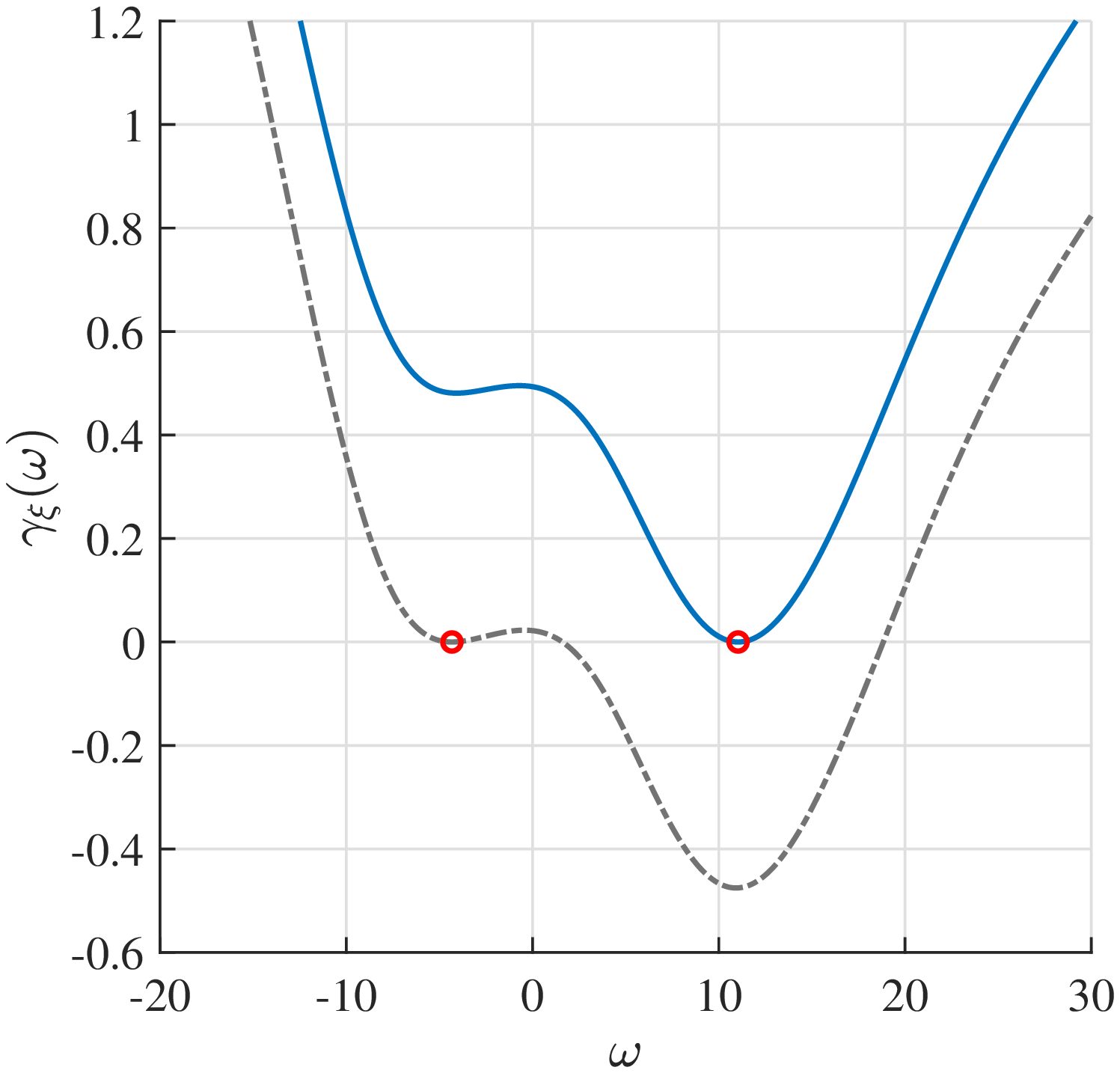}
		\caption{Random (cont.)}
		\label{fig:hec_rand}
	\end{subfigure}
	\begin{subfigure}{.329\textwidth}
		\centering
		\includegraphics[scale=\threescale,trim=0mm 0mm 0mm 0mm,clip]{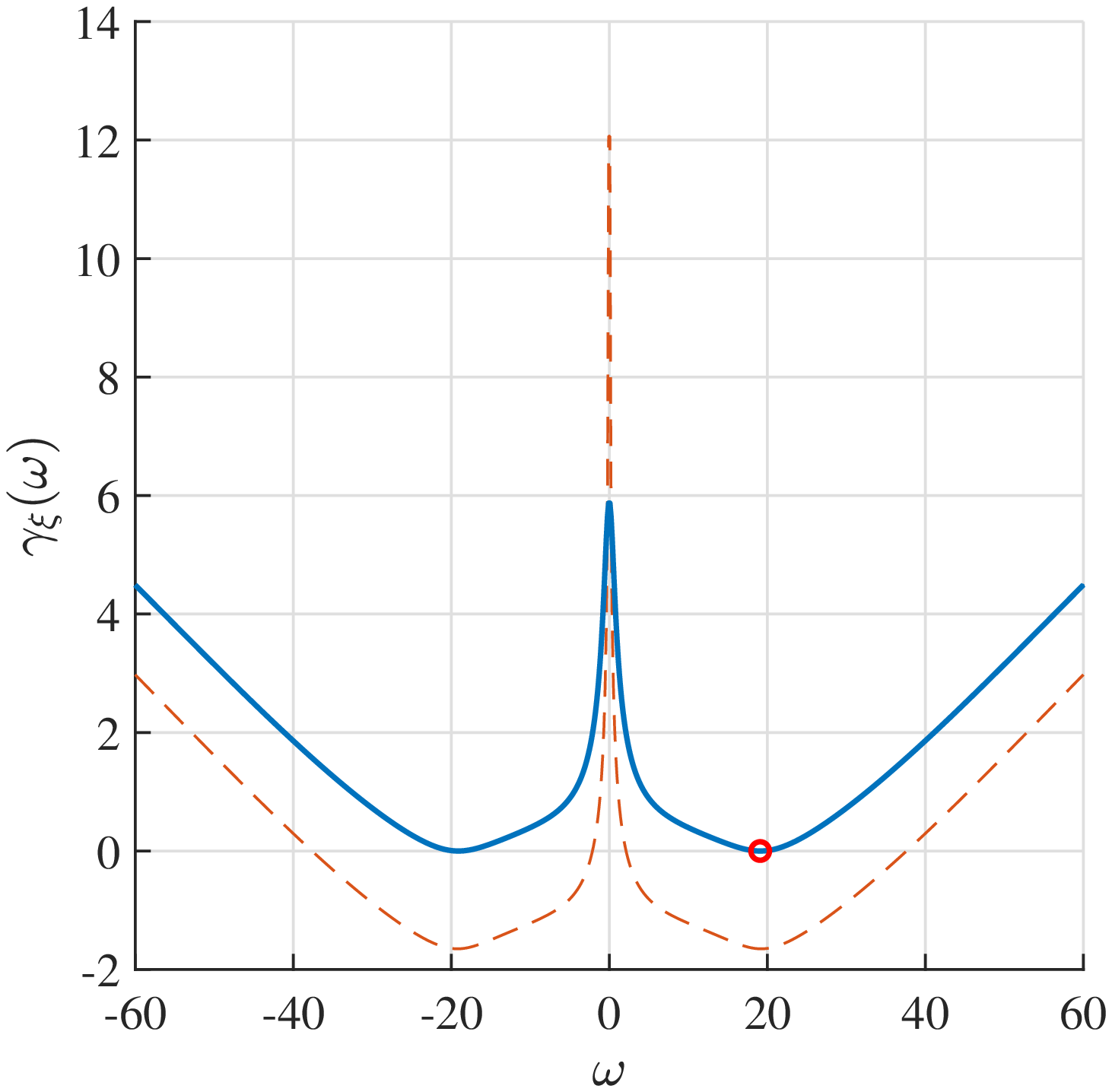}
		\caption{RLC (cont.)}
		\label{fig:hec_rlc}
	\end{subfigure}
	\begin{subfigure}{.329\textwidth}
		\centering
		\includegraphics[scale=\threescale,trim=0mm 0mm 0mm 0mm,clip]{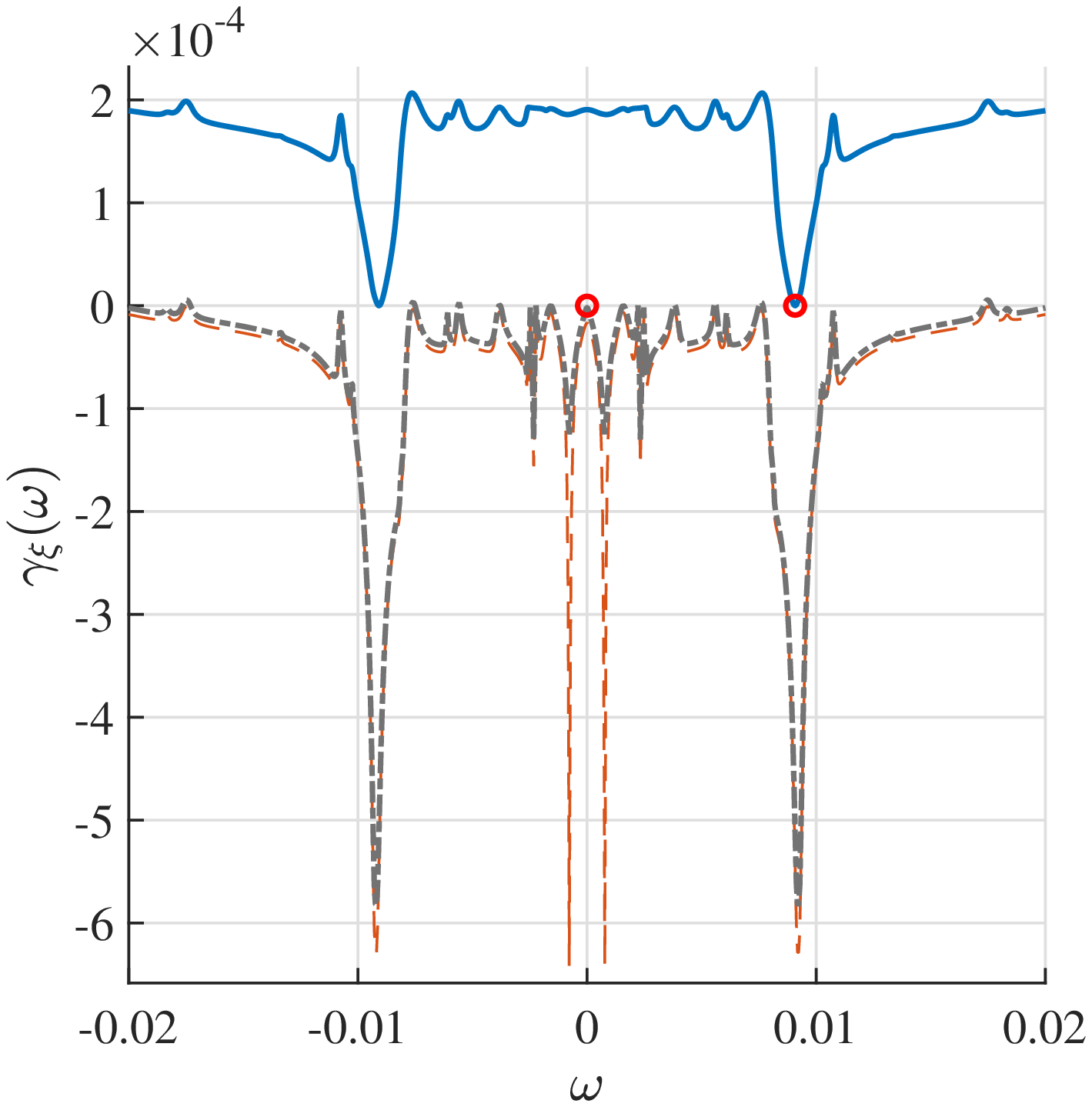}
		\caption{ISS (disc.)}
		\label{fig:hec_iss}
	\end{subfigure}
\caption{The pseudoroots $(\tilde \xi, \tilde \omega)$ and corresponding estimates $\tilde \xi$ for $\Xi$ computed
by HEC until convergence with $\tilde \xi \approx \Xi$.
}
\label{fig:hec_ex}
\end{figure}

While we have established that MP converges at least superlinearly, 
an examination of its iterates (not shown) seems to indicate that it too may converge quadratically
like HEC.
However, as demonstrated by Random and ISS, where MP respectively required 14 and 15 iterations,
MP can incur many iterations before it gets near its faster convergence regime.  
The key problem on these examples is that MP chooses the largest interval where~$\gamma_{\xi}$
is negative to determine how to reduce estimate $\xi$.
But this can be a particularly bad strategy if~$\gamma_{\xi}$ has a zero very far away from the origin, 
as is the case for both Random and ISS when~$\xi \approx \Xiub$.
While one could consider altering this strategy to improve performance, 
such an MP variant would still be slower than HEC and 
also still have the aforementioned numerical issues.

\begin{figure}[t]
	\begin{subfigure}{.49\textwidth}
		\centering
		\includegraphics[scale=\twoscale,trim=0mm 0mm 0mm 0mm,clip]{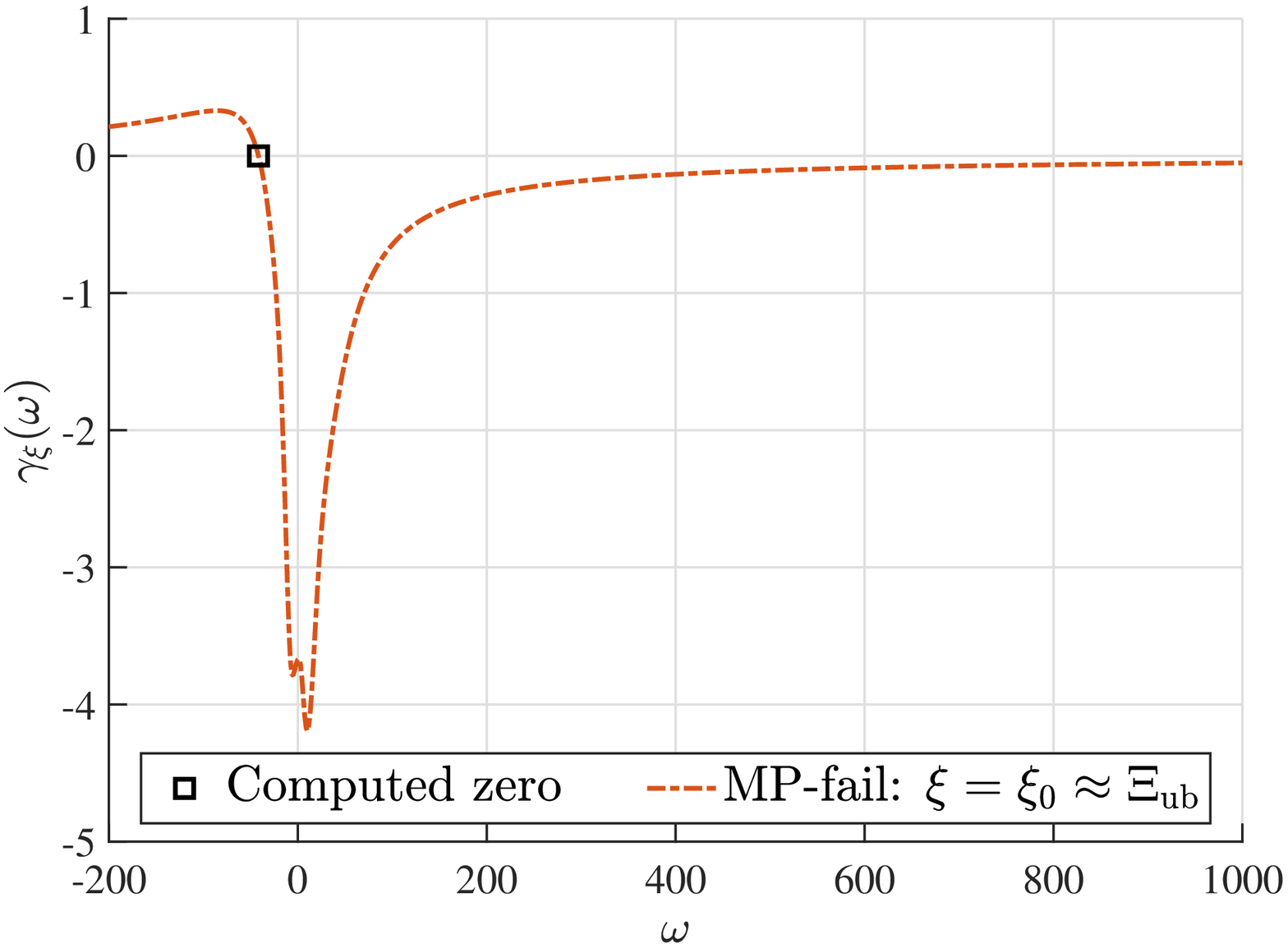}
		\caption{Random (cont.)}
		\label{fig:mp_fail}
	\end{subfigure}
	\begin{subfigure}{.49\textwidth}
		\centering
		\includegraphics[scale=\twoscale,trim=0mm 0mm 0mm 0mm,clip]{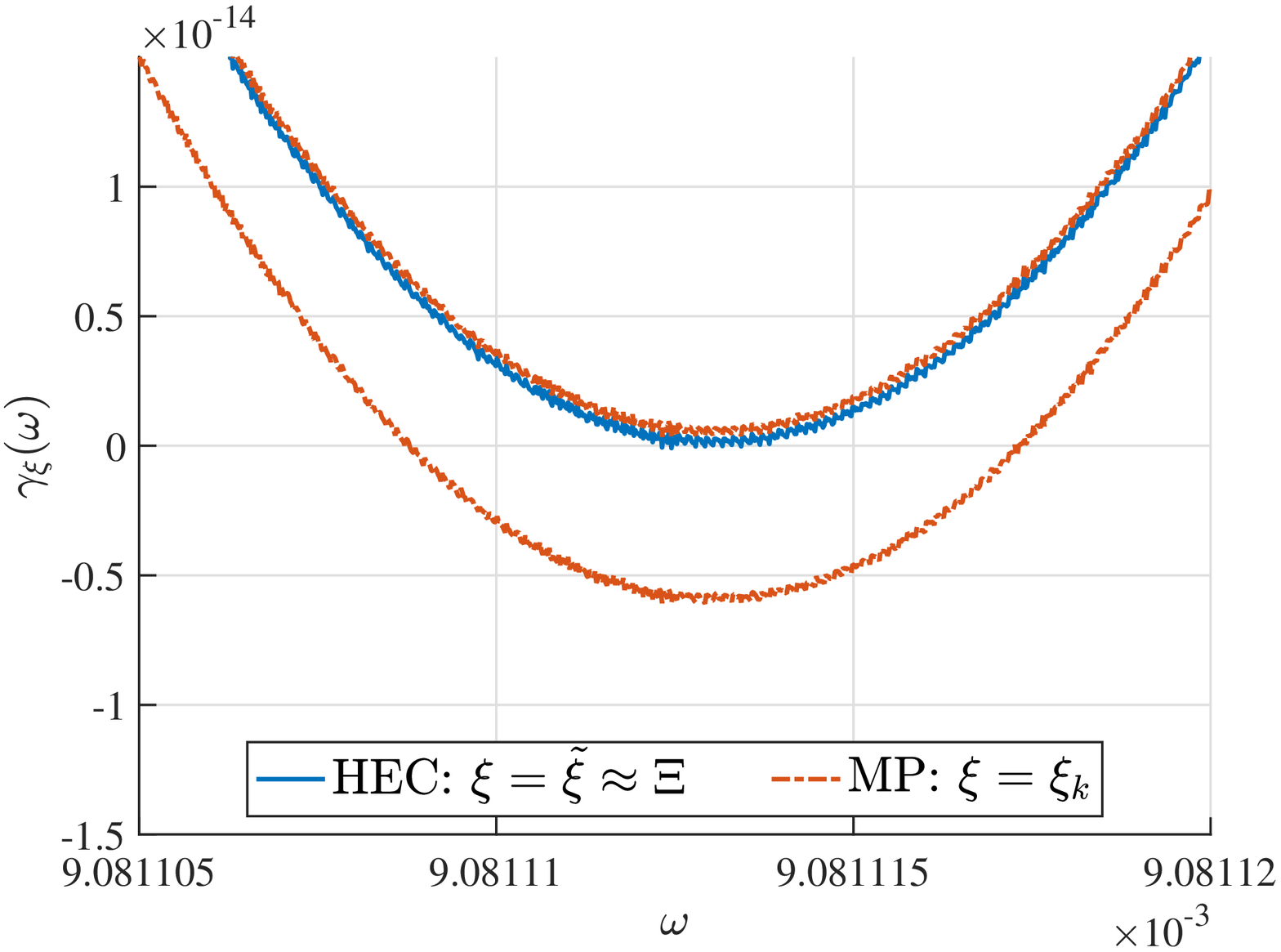}
		\caption{ISS (disc.)}
		\label{fig:mp_iss}
	\end{subfigure}
\caption{On the left, MP was initialized with $\xi_0 =  \Xiub - |\Xiub| 10^{-5}$.
In this case, $\gamma_{\xi_0}$ has two zeros at approximately $-41.6$ and 
$460600.9$, but the latter is not detected due to rounding errors when computing 
the eigenvalues of~$M_{\xi_0} - \lambda N$.
Consequently, MP erroneously terminates at $\xi_0$ with no digits of accuracy
but does converge properly when initialized with $\xi_0 =  \Xiub - |\Xiub| 10^{-4}$.
On the right, we see that eigenvalue computations used in MP 
to compute the smallest roots of $\gamma_{\omega}$ incurs
more rounding errors than our HEC-based approach.  See~\cite[section~9.2]{BenM19}
for an example where half of the precision can be lost when solving
root problems using these eigenvalue techniques.
}
\label{fig:mp_issues}
\end{figure}

\section{Conclusion}
\label{sec:conclusion}
By generalizing the HEC algorithm, we have presented faster and more numerically robust algorithms to compute $\Xi$,
 the extremal real value for which a given parametric linear time-invariant system is passive, 
a problem which is linked to maximizing the passivity radius.
Our new methods outperform the existing algorithms of Mehrmann and Van Dooren,
and for large-scale problems, when using sparse eigenvalue solvers, 
HEC can be used by itself to efficiently estimate~$\Xi$,
which the earlier methods cannot do.
We hope that our generalization of HEC, its convergence guarantees, and 
identification of root-max problems will also help facilitate new fast and robust numerical methods
for other quantities, for small- and large-scale problems.

While we have established local rates of convergence for our new methods (at least quadratic) and the earlier methods of Mehrmann and Van Dooren~\cite{MehV20,MehV20a} (at least superlinear),
one thing that remains unaddressed is the question of global convergence, \ie do these methods 
have unconditional convergence to $\Xi$? 
We believe that they do, but a potential 
wrinkle towards proving this is that the measure of the set of intervals where $\gamma_\xi$ is negative 
(see~\cref{cor:zeros_cont,cor:zeros_disc})
is not always decreasing. 
Consequently, computing~$\Xi$ is a fundamentally different problem than 
maximizing a one-variable function using the level-set technique
of Boyd and Balakrishnan~\cite{BoyB90}.

\footnotesize
\bibliographystyle{alpha} 
\bibliography{csc,mor,software}

\begin{thebibliography}{GGMO17}

\bibitem[BB90]{BoyB90}
S.~Boyd and V.~Balakrishnan.
\newblock A regularity result for the singular values of a transfer matrix and
  a quadratically convergent algorithm for computing its {${L}_\infty$-norm}.
\newblock {\em Systems Control Lett.}, 15(1):1--7, 1990.

\bibitem[BBMX99]{BenBMX99b}
P.~Benner, R.~Byers, V.~Mehrmann, and H.~Xu.
\newblock Numerical methods for linear-quadratic and {$H_\infty$} control
  problems.
\newblock In G.~Picci and D.~S. Gilliam, editors, {\em Dynamical Systems,
  Control, Coding, Computer Vision: New Trends, Interfaces, and Interplay},
  volume~25 of {\em Progress in Systems and Control Theory}, pages 203--222.
  Birkh{\"a}user, Basel, 1999.

\bibitem[BBMX02]{BenBMetal02}
P.~Benner, R.~Byers, V.~Mehrmann, and H.~Xu.
\newblock Numerical computation of deflating subspaces of
  skew-{H}amiltonian/{H}amiltonian pencils.
\newblock {\em {SIAM} J. Matrix Anal. Appl.}, 24(1):165--190, 2002.

\bibitem[BGVD20]{morBenGV20}
P.~Benner, P.~Goyal, and P.~Van~Dooren.
\newblock Identification of port-{H}amiltonian systems from frequency response
  data.
\newblock {\em Systems Control Lett.}, 143:104741, 2020.

\bibitem[BLO03]{BurLO03}
J.~V. Burke, A.~S. Lewis, and M.~L. Overton.
\newblock Robust stability and a criss-cross algorithm for pseudospectra.
\newblock {\em {IMA} J. Numer. Anal.}, 23(3):359--375, 2003.

\bibitem[BM19]{BenM19}
P.~Benner and T.~Mitchell.
\newblock Extended and improved criss-cross algorithms for computing the
  spectral value set abscissa and radius.
\newblock {\em {SIAM} J. Matrix Anal. Appl.}, 40(4):1325--1352, 2019.

\bibitem[DES82]{DemES82}
R.~S. Dembo, S.~C. Eisenstat, and T.~Steihaug.
\newblock Inexact {N}ewton methods.
\newblock {\em {SIAM} J. Numer. Anal.}, 19(2):400--408, 1982.

\bibitem[Fas02]{Fas02}
H.~Fassbender.
\newblock {\em Symplectic Methods for the Symplectic Eigenvalue Problem}.
\newblock Kluwer Academic Publishers, Dordrecht, NL, 2002.

\bibitem[FGL20]{FazGL20}
A.~Fazzi, N.~Guglielmi, and C.~Lubich.
\newblock Finding the nearest passive or non-passive system via hamiltonian
  eigenvalue optimization.
\newblock e-print arXiv:2010.15954, arXiv, October 2020.
\newblock math.NA.

\bibitem[GGMO17]{GugGMetal17}
N.~Guglielmi, M.~G{\"u}rb{\"u}zbalaban, T.~Mitchell, and M.~L. Overton.
\newblock Approximating the real structured stability radius with
  {F}robenius-norm bounded perturbations.
\newblock {\em {SIAM} J. Matrix Anal. Appl.}, 38(4):1323--1353, 2017.

\bibitem[GGO13]{GugGO13}
N.~Guglielmi, M.~G{\"u}rb{\"u}zbalaban, and M.~L. Overton.
\newblock Fast approximation of the {$H_\infty$} norm via optimization over
  spectral value sets.
\newblock {\em {SIAM} J. Matrix Anal. Appl.}, 34(2):709--737, 2013.

\bibitem[Glo84]{morGlo84}
K.~Glover.
\newblock All optimal {H}ankel-norm approximations of linear multivariable
  systems and their {L}$^\infty$-error norms.
\newblock {\em Internat. J. Control}, 39(6):1115--1193, 1984.

\bibitem[GS18]{GilS18}
N.~Gillis and P.~Sharma.
\newblock Finding the nearest positive-real system.
\newblock {\em {SIAM} J. Numer. Anal.}, 56(2):1022--1047, 2018.

\bibitem[HP90a]{HinP90}
D.~Hinrichsen and A.~J. Pritchard.
\newblock A note on some differences between real and complex stability radii.
\newblock {\em Systems Control Lett.}, 14(5):401--408, 1990.

\bibitem[HP90b]{HinP90a}
D.~Hinrichsen and A.~J. Pritchard.
\newblock Real and complex stability radii: a survey.
\newblock In {\em Control of uncertain systems ({B}remen, 1989)}, volume~6 of
  {\em Progr. Systems Control Theory}, pages 119--162. Birkh{\"a}user Boston,
  Boston, MA, 1990.

\bibitem[HP05]{HinP05}
D.~Hinrichsen and A.~J. Pritchard.
\newblock {\em Mathematical Systems Theory I}.
\newblock Springer-Verlag, Berlin, 2005.

\bibitem[KSW10]{KreSW09}
D.~Kressner, C.~Schr\"{o}der, and D.~Watkins.
\newblock Implicit {QR} algorithms for palindromic and even eigenvalue
  problems.
\newblock {\em Numer. Algorithms}, 51:209--238, 2010.

\bibitem[Lan64]{Lan64}
P.~Lancaster.
\newblock On eigenvalues of matrices dependent on a parameter.
\newblock {\em Numer. Math.}, 6:377--387, 1964.

\bibitem[Lau81]{Lau81}
A.~Laub.
\newblock Efficient multivariable frequency response computations.
\newblock {\em {IEEE} Trans. Autom. Control}, 26(2):407--408, April 1981.

\bibitem[Meh91]{Meh91}
V.~Mehrmann.
\newblock {\em The Autonomous Linear Quadratic Control Problem, Theory and
  Numerical Solution}.
\newblock Number 163 in Lecture Notes in Control and Information Sciences.
  Springer-Verlag, Heidelberg, July 1991.

\bibitem[Mit14]{Mit14}
T.~Mitchell.
\newblock {\em Robust and efficient methods for approximation and optimization
  of stability measures}.
\newblock PhD thesis, New York University, New York, NY, USA, September 2014.

\bibitem[MO16]{MitO16}
T.~Mitchell and M.~L. Overton.
\newblock Hybrid expansion-contraction: a robust scaleable method for
  approximating the {$H_\infty$} norm.
\newblock {\em {IMA} J. Numer. Anal.}, 36(3):985--1014, 2016.

\bibitem[MO22]{MitO22}
T.~Mitchell and M.~L. Overton.
\newblock On properties of univariate max functions at local maximizers.
\newblock {\em Optim. Lett.}, 2022.

\bibitem[MVD20a]{MehV20a}
V.~Mehrmann and P.~Van~Dooren.
\newblock Optimal robustness of passive discrete-time systems.
\newblock {\em {IMA} J. Math. Control Inform.}, 37(4):1248--1269, July 2020.

\bibitem[MVD20b]{MehV20}
V.~Mehrmann and P.~M. Van~Dooren.
\newblock Optimal robustness of port-{H}amiltonian systems.
\newblock {\em {SIAM} J. Matrix Anal. Appl.}, 41(1):134--151, 2020.

\bibitem[NW99]{NocW99}
J.~Nocedal and S.~J. Wright.
\newblock {\em {N}umerical {O}ptimization.}
\newblock Springer, New York, 1999.

\bibitem[OVD05]{OveV05}
M.~L. Overton and P.~Van~Dooren.
\newblock On computing the complex passivity radius.
\newblock In {\em Proceedings of the 44th IEEE Conference on Decision and
  Control}, volume~49, pages 7960--7964, December 2005.

\bibitem[OW95]{OveW95}
M.~L. Overton and R.~S. Womersley.
\newblock Second derivatives for optimizing eigenvalues of symmetric matrices.
\newblock {\em {SIAM} J. Matrix Anal. Appl.}, 16(3):697--718, 1995.

\bibitem[RW98]{RocW98}
R.~T. Rockafellar and R.~J.-B. Wets.
\newblock {\em Variational analysis}, volume 317 of {\em Grundlehren der
  Mathematischen Wissenschaften [Fundamental Principles of Mathematical
  Sciences]}.
\newblock Springer-Verlag, Berlin, 1998.

\bibitem[{Van}85]{Van85a}
C.~F. {Van~Loan}.
\newblock How near is a stable matrix to an unstable matrix?
\newblock In {\em Linear algebra and its role in systems theory ({B}runswick,
  {M}aine, 1984)}, volume~47 of {\em Contemp. Math.}, pages 465--478. Amer.
  Math. Soc., Providence, RI, 1985.

\bibitem[VD81]{Van81b}
P.~Van~Dooren.
\newblock The generalized eigenstructure problem in linear system theory.
\newblock {\em {IEEE} Trans. Autom. Control}, 26:111--129, 1981.

\bibitem[Wil71]{Wil71}
J.~C. Willems.
\newblock Least squares stationary optimal control and the algebraic {R}iccati
  equation.
\newblock {\em {IEEE} Trans. Autom. Control}, 16:621--634, 1971.

\bibitem[Wil72]{Wil72a}
J.~C. Willems.
\newblock Dissipative dynamical systems. {II}. {L}inear systems with quadratic
  supply rates.
\newblock {\em Arch. Rational Mech. Anal.}, 45:352--393, 1972.

\bibitem[Xu06]{Xu06}
H.~Xu.
\newblock On equivalence of pencils from discrete-time and continuous-time
  control.
\newblock {\em Linear Algebra Appl.}, 414:97--124, 2006.

\bibitem[ZDG96]{ZhoDG96}
K.~Zhou, J.~C. Doyle, and K.~Glover.
\newblock {\em Robust and Optimal Control}.
\newblock Prentice-Hall, Upper Saddle River, NJ, 1996.

\end{thebibliography}
\end{document}